\documentclass{amsart}

\usepackage{adjustbox}
\usepackage{amssymb}
\usepackage{amsthm}
\usepackage{enumitem}
\usepackage{eucal}
\usepackage[pagebackref=true]{hyperref}
\usepackage{mathtools}
\usepackage{stmaryrd}
\usepackage{tikz-cd}

\setlist[enumerate]{leftmargin=1cm}

\hypersetup{
	bookmarksnumbered=true,
	colorlinks=true,
	linkcolor=blue,
	citecolor=blue,
	filecolor=blue,
	menucolor=blue,
	urlcolor=blue,
	bookmarksopen=true,
	bookmarksdepth=2,
	pageanchor=true}

\makeatletter
\@namedef{subjclassname@2020}{\textup{2020} Mathematics Subject Classification}
\makeatother

\numberwithin{equation}{section}
\swapnumbers

\theoremstyle{plain}
\newtheorem{theorem}[equation]{Theorem}
\newtheorem{proposition}[equation]{Proposition}
\newtheorem{lemma}[equation]{Lemma} 
\newtheorem{corollary}[equation]{Corollary}

\theoremstyle{definition}
\newtheorem{definition}[equation]{Definition}
\newtheorem{example}[equation]{Example}

\newtheorem{chunk}[equation]{}

\theoremstyle{remark}
\newtheorem{remark}[equation]{Remark} 

\newtheorem*{ack}{Acknowledgements}
\newtheorem*{funding}{Funding}

\newcommand{\bbN}{\mathbb N} 

\newcommand{\colonequals}{\coloneqq}
\newcommand{\Coker}{\operatorname{Coker}}

\newcommand{\dbcat}[1]{\mathbf{D}^{\mathrm{b}}(\mod #1)}
\newcommand{\dcat}[1]{\mathbf{D}(\Mod #1)}
\newcommand{\dsing}{\mathbf{D}_{\mathrm{sg}}}
\newcommand{\depth}{\operatorname{depth}}

\newcommand{\End}{\operatorname{End}}
\newcommand{\equalscolon}{\eqqcolon}
\newcommand{\Ext}{\operatorname{Ext}}

\newcommand{\fm}{\mathfrak{m}} 
\newcommand{\fp}{\mathfrak{p}}
\newcommand{\fibre}[2]{ {#1}_{k(#2)}}

\newcommand{\gldim}{\operatorname{gl\,dim}}

\newcommand{\Hom}{\operatorname{Hom}}

\renewcommand{\Im}{\operatorname{Im}}
\newcommand{\ind}[1]{\operatorname{ind}{#1}}
\newcommand{\inj}{\operatorname{inj}}

\newcommand{\injdim}{\operatorname{inj\,dim}}
\newcommand*{\intref}[2]{\def\tmp{#1}\ifx\tmp\empty\hyperref[#2]{\ref*{#2}}\else\hyperref[#2]{#1~\ref*{#2}}\fi}

\newcommand{\iso}{\xrightarrow{\raisebox{-.4ex}[0ex][0ex]{$\scriptstyle{\sim}$}}}

\newcommand{\longiso}{\xrightarrow{\ \raisebox{-.4ex}[0ex][0ex]{$\scriptstyle{\sim}$}\ }}

\newcommand{\Ker}{\operatorname{Ker}}

\newcommand{\lat}[1]{\operatorname{lat}(#1)}
\newcommand{\Loc}{\operatorname{Loc}}
\newcommand{\lotimes}{\otimes^{\mathrm L}}
\newcommand{\lra}{\longrightarrow}
\newcommand{\Mat}{\operatorname{Mat}}
\newcommand{\mcm}[1]{\operatorname{mCM}(#1)}

\renewcommand{\mod}{\operatorname{mod}}
\newcommand{\Mod}{\operatorname{Mod}}
\newcommand{\op}[1]{{#1}^{\mathrm{op}}}
\newcommand{\pdim}{\operatorname{proj\,dim}}

\newcommand{\proj}{\operatorname{proj}}
\newcommand{\rad}[1]{\operatorname{rad}(#1)}
\newcommand{\rank}{\operatorname{rank}}
\newcommand{\RHom}{\operatorname{RHom}}
\newcommand{\sHom}{\underline{\Hom}}
\newcommand{\smcm}[1]{\underline{\operatorname{mCM}}(#1)}
\newcommand{\Spec}{\operatorname{Spec}}
\newcommand{\supp}{\operatorname{supp}}
\newcommand{\syz}{\Omega}
\newcommand{\Thick}{\operatorname{Thick}}
\newcommand{\tor}{\operatorname{tor}}
\newcommand{\Tor}{\operatorname{Tor}}
\newcommand{\Tr}[1]{\operatorname{Tr}(#1)}
\newcommand{\wh}{\widehat}

\title[A class of Gorenstein algebras and their dualities]{A class of Gorenstein algebras \\ and their dualities}

\author[W. Gnedin, S. B. Iyengar and H. Krause]{Wassilij Gnedin, Srikanth  B. Iyengar and Henning Krause}

\address{Wassilij Gnedin\\
	Institut f\"ur Mathematik\\ 
	Universit\"at Paderborn\\
	33098 Paderborn\\
	Germany\\
	\href{mailto:gnedin@math.uni-paderborn.de}{gnedin@math.uni-paderborn.de}}

\address{Srikanth B. Iyengar\\
	Department of Mathematics\\
	University of Utah\\
	Salt Lake City, UT 84112\\
	U.S.A.\\
\href{mailto:Srikanth.B.Iyengar@utah.edu}{Srikanth.B.Iyengar@utah.edu}}

\address{Henning Krause\\ 
	Fakult\"at f\"ur Mathematik\\ 
	Universit\"at Bielefeld\\
	33501 Bielefeld\\ 
	Germany\\
	\href{mailto:hkrause@math.uni-bielefeld.de}{hkrause@math.uni-bielefeld.de}}

\begin{document}

\begin{abstract}
	In the recent paper ``The Nakayama functor and its completion for Gorenstein algebras'', a class of Gorenstein algebras over commutative noetherian rings was introduced, and duality theorems for various categories of representations were established. The manuscript on hand provides more context to the results presented in the aforementioned work, identifies new classes of Gorenstein algebras, and explores their behaviour under standard operations like taking tensor products and tilting.
\end{abstract}


\maketitle

\section{Introduction}

The concept of a Gorenstein commutative ring was introduced by Bass~\cite{Bass:1963} almost sixty years ago. It continues to play a central role in commutative algebra and algebraic geometry, in particular in investigations concerning duality. Following the work of Auslander and Reiten~\cite{Auslander:1978a, Auslander--Reiten:1991} and Buchweitz~\cite{Buchweitz:1986} it was recognised that such rings are also special from the perspective of representation theory, and that this is not limited to the commutative context. Indeed, there are various generalisations of the Gorenstein property that apply to general non-commutative rings. Especially in the world of finite dimensional  algebras, the representation theory of Gorenstein algebras remains an active area of research. 

Motivated by their work  with Benson and Pevtsova~\cite{Benson--Iyengar--Krause--Pevtsova:2019a,Benson--Iyengar--Krause--Pevtsova:2020a} on local duality for finite group schemes and more general finite dimensional Gorenstein algebras,  the second and third authors  introduced in  \cite{Iyengar--Krause:2021} a notion of Gorenstein algebras that encompasses most Gorenstein algebras that arise naturally, including all the commutative ones. The goal was to establish duality results,  akin to those established by Grothendieck for commutative rings, for suitable categories of representations over such Gorenstein algebras. These had been obtained by Buchweitz~\cite{Buchweitz:1986} in a more limited context; we refer interested readers to the introductory section of \cite{Iyengar--Krause:2021} for connections to earlier literature. 

Some of these dualities were discussed in the lecture series delivered by the second author at the ICRA workshop pertaining to this volume. In the present manuscript the focus is on other aspects of the theory, including a different perspective on Gorenstein algebras, as flat families of finite dimensional Gorenstein algebras; on natural constructions that preserve this property, notably tilting; and on describing various natural occurring examples of Gorenstein algebras, mostly notably gentle Gorenstein algebras.   While this entails recalling basic results from \cite{Iyengar--Krause:2021}, there are also new results presented here, and not just about Gorenstein algebras. Another aspect that distinguishes the two manuscripts is that the current one is focussed mainly on the representation theory of finitely generated modules, whereas \cite{Iyengar--Krause:2021} is not limited to these. Indeed one has to contend with infinitely generated modules  even to formulate the duality statements presented there.

This article is organised as follows.
Section~\ref{se:commutative-gorenstein} recalls 
basic notions related to Gorenstein rings.
Section~\ref{se:finite-projective} gives first examples of properties of complexes over a finite projective algebra which can be studied in terms of its fibres.
The last topic is continued in Sections~\ref{se:tilting} and~\ref{se:lifting-problems} in connection with tilting theory. Section~\ref{se:Gorenstein-algebras} introduces Gorenstein algebras and shows the invariance of this notion under certain operations.
Section~\ref{se:maximal-cohen-macaulay-modules} provides duality theorems for a special type of Gorenstein algebras, in particular, one 
 for the singularity category of such an algebra (Theorem~\ref{theorem:serre-duality}).
The ultimate goal of Sections~\ref{se:gentle-algebras},
\ref{se:serre-gentle} and~\ref{se:singularity-categories} 
is to make the notions of the last theorem concrete in the context of gentle algebras. Each section requires only a few statements from its predecessors.

\section{Commutative Gorenstein rings}\label{se:commutative-gorenstein}
Throughout this manuscript $R$ is a commutative noetherian ring.  The main examples to keep in mind are coordinate rings of algebraic varieties, and rings coming from number theory.  In this section we recall the salient features of the theory of injective dimension of modules over commutative rings, leading to the notion of a Gorenstein ring, and duality properties of modules over such rings. The interested reader can find much more in Bass' paper~\cite{Bass:1963}, which started the subject of Gorenstein rings, and  in the textbook by Bruns and Herzog~\cite[Chapter~3]{Bruns--Herzog:1998a}.

\begin{chunk}
	\label{ch:injdim-test}
	Let $(R,\fm,k)$ be a commutative local ring; the notation means that $\fm$ is the maximal ideal of $R$ and $k$ its residue field, $R/\fm$. Set $d\colonequals \dim R$, the Krull dimension on $R$; this is finite since $R$ is local
	and noetherian. For any finitely generated $R$-module $M$ there are equivalences
	\begin{align*}
		\injdim_RM <\infty 
		&\iff  \Ext_R^i(-,M)=0 \quad \text{for $i\gg 0$} \\
		& \iff  \Ext_R^i(k,M)=0 \quad \text{for $i\gg 0$} \\
		&\iff  \Ext_R^i(k,M)=0 \quad \text{for some $i> d$} \\
		&\iff  \Ext_R^{d+1}(k,M)=0\,. 
	\end{align*}
	The first equivalence is essentially by definition; the second is by a standard noetherian induction, and requires that $M$ is finitely generated; see \cite[Proposition~3.1.14]{Bruns--Herzog:1998a}. The next equivalence is a theorem of Fossum, Foxby, Griffith and Reiten~\cite[Theorem~1.1]{Fossum--Foxby--Griffith--Reiten:1975},  and depends on the existence of big Cohen--Macaulay modules. The last implication is a consequence of the following result of Ischebeck~\cite[Exercise~3.1.24]{Bruns--Herzog:1998a}: When $\injdim_RM$ is finite, any finitely generated $R$-module $N$ satisfies
	\[
	\sup\{i\in\bbN\mid \Ext^i_R(N,M)\neq 0\} = \depth R - \depth_RN\,.
	\]
	We recall that the \emph{depth} of an $R$-module $N$ is
	\[
	\depth_RN\colonequals \inf\{i\in \bbN \mid \Ext^i_R(k,N)\ne 0\}\,.
	\]
	
	In summary we get that
	\begin{align*}
		\injdim_RM <\infty \iff \injdim_RM\le \depth R \iff \injdim_R M = \depth R\,.
	\end{align*}
	Bass had conjectured that a local ring $R$ is Cohen--Macaulay when it has a nonzero finitely generated module of finite injective dimension. This conjecture was one of the ``homological conjectures" in local algebra, and was finally settled by P.~Roberts, as a consequence of his proof of the New Intersection Theorem; see \cite{Roberts:1980, Roberts:1987}. Given this, one can replace $\depth R$ by $\dim R$ in the equivalences stated above.
\end{chunk}

\begin{chunk}
	\label{ch:Gor-local}
	A commutative local ring $R$ is \emph{Gorenstein} if $\injdim R<\infty$. It follows from the discussion in \ref{ch:injdim-test} that $R$ is Gorenstein if and only if $\injdim R =\dim R$. Then
	\[
	\Ext_R^i(k,R) =
	\begin{cases}
		k & i = \dim R\,,\\ 
		0 & \text{otherwise}\,.
	\end{cases}
	\]
	That is to say $\RHom_R(k,R)\cong \Sigma^{-\dim R} k$.
\end{chunk}

\begin{chunk}
	\label{ch:support}
	Let $R$ be a commutative ring, not necessarily local. We write $\Spec R$ for the spectrum of $R$, consisting of the prime ideals in $R$ with the Zariski topology. The \emph{support} of a finitely generated $R$-module $M$ is the subset
	\[
	\supp_RM \colonequals \{\fp\in\Spec R\mid M_\fp \ne 0\}\,.
	\]
	It is the closed subset $V(I)$, where $I$ is the annihilator ideal of $M$, namely, the kernel of the homothety map $R\to \End_R(M)$.
\end{chunk}

\begin{chunk}
	\label{ch:Gor-global}
	A commutative  ring $R$ is \emph{Gorenstein} if it is noetherian and the local ring $R_\fp$ is Gorenstein, in the sense of \ref{ch:Gor-local}, for each $\fp$ in $\Spec R$; equivalently, for each $\fp$ in $\max(\Spec R)$. When this holds $\injdim R = \dim R$, and then $R$  has finite injective dimension precisely when $\dim R$ is finite. In particular $R$ is Artinian and Gorenstein if and only if $\injdim R =0$; equivalently, when $R$ is self-injective.
	
	Prominent examples include regular rings (which are, by definition, rings locally of finite global dimension),  certain determinantal rings, and, more generally, rings of invariants of certain finite groups; see \cite{Bruns--Herzog:1998a}.
\end{chunk}

\begin{chunk}
	\label{ch:Nagata}
	Nagata~\cite[Appendix~A1]{Nagata:1975} constructed a commutative noetherian ring $R$ such that $\gldim R_\fp$ is finite for each $\fp$ in $\Spec R$ (so $R$ is a regular ring), and $\dim R=\infty$.  This ring is, in particular, Gorenstein, which makes the point that  $R$ Gorenstein does not mean $\injdim R$ is finite.  In other words, even in the world of commutative rings, Gorenstein does not mean Iwanaga--Gorenstein; see Section~\ref{se:Gorenstein-algebras}.
\end{chunk}

The following theorem of Goto~\cite{Goto:1982} often allows one to extend results known for rings of finite injective dimension, to the more general class of Gorenstein rings. 

\begin{theorem}
	\label{th:goto}
	Let $R$ be a commutative Gorenstein ring. For any $M$ in $\dbcat R$ one has $\Ext_R^i(M,R)=0$ for $i\gg 0$.  \qed
\end{theorem}

The remarkable aspect of this result is that $\Ext_R^*(M,R)$ is bounded although the injective dimension of $R$ need not be finite.  Goto's theorem easily leads to the following theorem, which also goes by the name ``local duality theorem", for Grothendieck's local duality can be deduced from it readily; see \cite{Roberts:1980}.

\begin{theorem}
	\label{th:goto+}
	When $R$ is a commutative Gorenstein ring  $\RHom_R(-,R)$ induces a triangle equivalence $\op{\dbcat R} \iso \dbcat R$. \qed
\end{theorem}

\section{Finite projective algebras}\label{sec:finite-projective}
\label{se:finite-projective}

Let $R$ be a commutative noetherian ring.  Following Gabriel~\cite{Gabriel:1962}, by a \emph{finite $R$-algebra} $A$ we mean that $A$ is an associative $R$-algebra that is finitely generated as an $R$-module. In this case, $A$ is noetherian on either side, its centre $Z(A)$ is a noetherian ring, and $A$ is a finite $Z(A)$-algebra.  Thus for any $A$-complexes $M,N$ in $\dbcat A$, the bounded derived category of finitely generated (left) $A$-modules, $\Ext_A^i(M,N)$ is finitely generated as an $R$-module, and hence also as a $Z(A)$-module, for each integer $i$.

We say $A$ is a  \emph{projective} $R$-algebra to mean that $A$ is an $R$-algebra that is projective as an $R$-module. Thus a finite projective $R$-algebra is an $R$-algebra that is both finite and projective as an $R$-module. Examples include the group algebra, $RG$, over $R$ of a finite group $G$ and $RQ$, the quiver algebra of a finite quiver $Q$ without loops and cycles. 

\subsection*{Support and faithfulness}
Let $A$ be a finite $R$-algebra. The support of $A$ as an $R$-module, in the sense of \ref{ch:support}, has another interpretation.

\begin{chunk}
	\label{ch:specA}
	Gabriel~\cite{Gabriel:1962} introduced the \emph{spectrum} of $A$, denoted $\Spec A$, to be the collection of two-sided prime ideals in $A$. 
	A proper two-sided ideal $\fp$ of $A$ is a \emph{prime ideal} if for any two-sided ideals $I,J$ of $A$ with $IJ \subseteq \fp$ it follows that $I \subseteq \fp$ or $J \subseteq \fp$.
	Given such a $\fp$ its contraction $\fp\cap R$ to $R$ is again a prime ideal, so we get a map $\Spec A\to \Spec R$. 
	It can be shown that its image is precisely $\supp_RA$, see \cite[Chapter V, Proof of Proposition 11]{Gabriel:1962}.
\end{chunk}

This leads to the question: When is $\supp_RA = \Spec R$? This is connected to the faithfulness of $A$ as an $R$-module, as is explained below.

\begin{chunk}
	\label{ch:faithful}
	An $R$-module $M$ is \emph{faithful} if the homothety map $R\to \End_R(M)$ is one-to-one. An $R$-algebra $A$ is faithful as an $R$-module precisely when the structure map $R\to A$ is one-to-one. 
	
	When a finitely generated $R$-module $M$ is faithful, $\supp_RM=\Spec R$. The converse need not hold; for instance, if $I$ is a nilpotent ideal in $R$, then $\supp_R(R/I)=\Spec R$, but the $R$-module $R/I$ is not faithful when $I\ne 0$. 
\end{chunk}

\begin{lemma}
	Let $M$ be a finitely generated projective $R$-module.  Then $M$ is faithful if and only if $\supp_RM=\Spec R$. 
\end{lemma}

\begin{proof}
	We only have to verify that when $\supp_RM=\Spec R$, the $R$-module $M$ is faithful.  Let $I$ be kernel of the homothety map $\lambda\colon R\to \End_R(M)$. Since $M$ is projective, the $R_\fp$-module $M_\fp$ is free for each $\fp$ in $\Spec R$; it is also nonzero because $\supp_RM=\Spec R$. Thus the induced map
	\[
	\lambda_\fp\colon R_\fp \lra \End_R(M)_\fp \cong \End_{R_\fp}(M_\fp)
	\]
	is injective; the isomorphism holds because $M$ is finitely generated. Thus $I_\fp=0$. Since this holds for each prime ideal $\fp$, it follows that $I=0$, as desired.
\end{proof}

\begin{chunk}
	\label{ch:subring}
	Let $M$ be a nonzero finitely generated projective $R$-module. Consider the rank function 
	\[
	\Spec R\lra \bbN \qquad\text{where $\fp \mapsto \rank_{R_\fp}(M_\fp)$\,.}
	\]
	This is continuous for the Zariski topology on $\Spec R$ and discrete topology on $\bbN$, and hence constant on  connected subsets of $\Spec R$; see \cite[Chapter~III, Theorem~7.1]{Bass:1968}. Therefore when $\Spec R$ is connected, the $R$-module $M$, being projective and finitely generated, is faithful. In general, $R$ decomposes as $R'\times R''$, where $\Spec R'=\supp_RM$, with $M$ finitely generated even when viewed as an $R'$-module, and $R'' \cdot  M= 0$. 
	
	The upshot of this discussion is that if $A$ is a finite projective $R$-algebra, then by dropping to a ring factor $R'$ of $R$ if needed, we can assume that $A$ is also faithful as an $R$-module; equivalently, that $R$ is a subring of the centre of $A$. 
\end{chunk}

\subsection*{Fibre-wise criteria}
Let $A$ be a finite projective $R$-algebra. For each $\fp$ in $\Spec R$ the ring $A \otimes_R k(\fp)$ is a finite dimensional $k(\fp)$-algebra. One can thus view $A$ as a family of finite dimensional algebras parameterised  by the (spectrum of the) ring $R$.
Moreover any $A$-complex  $M$ defines the family of complexes
\[
\fibre M\fp \colonequals   M \lotimes_R k(\fp)   \qquad\text{for $\fp$ in $\Spec R$\,.}
\]
We identify $\fibre A\fp$ with the $k(\fp)$-algebra $A\otimes_R k(\fp)$, and view $\fibre M\fp$ as a complex over it. The following results validate this perspective.  

An $A$-complex $M$ is \emph{perfect} if it is quasi-isomorphic to a bounded complex of finitely generated projective $A$-modules.
\begin{proposition}
	\label{pr:perfection}
	Let $A$ be a finite projective $R$-algebra. For any $M\in \dbcat A$, the following conditions are equivalent.
	\begin{enumerate}
		\item \label{pr:per1}
		The $A$-complex $M$ is perfect;
		\item \label{pr:per2}
		The $\fibre A\fp$-complex $\fibre M\fp$ is perfect for each $\fp\in \supp_RA$;
		\item \label{pr:per3}
		The $\fibre A\fm$-complex $\fibre M\fm$ is perfect for each $\fm\in \max(\supp_RA)$.
	\end{enumerate}
\end{proposition}

\begin{proof}
	It is straightforward to verify that \eqref{pr:per1}$\Rightarrow$\eqref{pr:per2}$\Rightarrow$\eqref{pr:per3}.

	\eqref{pr:per3}$\Rightarrow$\eqref{pr:per1} Since $M$ is in $\dbcat A$, it suffices to verify that for each maximal ideal $\fm$ in $\supp_RA$ and simple $A_\fm$-module $L$ one has
	\[
	\Ext^i_{A_\fm}(M_\fm,L)=0 \qquad{i\gg 0}\,.
	\]
	This is by \cite[Theorem A.1.2]{Avramov--Iyengar:2021}. 
	Since $\fm A_\fm$ is contained in the Jacobson radical of $A_\fm$ according to \cite[Corollary~5.9]{Lam:1991}, one has $\fm L =0$ and the action of $A_\fm$ on $L$ factors through $\fibre A\fm$. Thus  standard adjunction yields the isomorphism below:
	\[
	\Ext^i_{A_\fm}(M_\fm,L) \cong \Ext_{\fibre A\fm}^i(\fibre M\fm,L) =0 \qquad {i\gg 0}\,.
	\]
	The equality holds because $\fibre M\fm$ is a perfect $\fibre A\fm$-complex. 
\end{proof}

A variation of the last proof yields a fibre-wise criterion for projectivity of certain $A$-modules. More precisely, the following statement holds, see \cite[Lemma~2.19]{Benson--Iyengar--Krause--Pevtsova:2022}.
\begin{lemma}\label{le:proj}
	Let $A$ be a finite projective $R$-algebra, and $M$ a finitely generated $A$-module.
	When $M$ is projective as an $R$-module, there is an equality
	\begin{align*}
		\pdim_A M = \sup\{ \pdim_{\fibre A\fp}{\fibre M\fp} \mid \fp \in \Spec R\}\,.
	\end{align*}  
	Moreover, it suffices to take the supremum over the maximal ideals in $R$.
	\qed
\end{lemma}

The following result due to Bass and Murthy, see ~\cite[Theorem~7.1]{Bass:1968} or \cite[Theorem~4.1]{Avramov--Iyengar--Lipman:2010}, is a key ingredient in the proof of  \cite[Theorem A.1.2]{Avramov--Iyengar:2021} used above. It plays  a role in the discussion below, and also later on.

\begin{theorem}
	\label{th:bass-murthy}
	Let $R$ be a commutative noetherian ring,  $A$ a finite $R$-algebra, and $M$ an $A$-complex in $\dbcat{A}$. Then $M$ is perfect if and only if the $A_\fp$-complex $M_\fp$ is perfect for each $\fp$ in  
	$\supp_RA$.
	\qed
\end{theorem}

It is worth noting that the analogous statement for injective dimensions does not hold in general; any Gorenstein ring that is not of finite injective dimension illustrates this; see \ref{ch:Nagata}.

For later use we record a few more  properties of $A$-complexes that can be detected on the fibres. Given a complex $N$ in $\dcat A$ we write $\Loc(N)$ and $\Thick(N)$ for the localising subcategory and the thick subcategory, respectively, generated by $N$.

\begin{theorem}
	\label{th:Neeman}
	Let $A$ be a finite projective $R$-algebra, and let $M,N$ be $A$-complexes. The following conditions are equivalent. 
	\begin{enumerate}
		\item \label{th:Neeman1}
		$M\in  \Loc(N)$ in $\dcat A$;
		\item \label{th:Neeman2}
		$\fibre M\fp\in \Loc(\fibre N{\fp})$ in $\dcat {\fibre A\fp}$ for each $\fp \in \supp_RA$.
	\end{enumerate}
\end{theorem}

\begin{proof}
	The  functor $\dcat A$ to $\dcat{\fibre A\fp}$ assigning $M\mapsto \fibre M\fp$  is exact and preserves coproducts. This yields that \eqref{th:Neeman1}$\Rightarrow$\eqref{th:Neeman2}.
	
	\eqref{th:Neeman2}$\Rightarrow$\eqref{th:Neeman1} The key input is a local to global principle for detecting membership in localising subcategories, from \cite[\S3]{Benson--Iyengar--Krause:2011a}. To that end we view $\dcat A$ as an $R$-linear triangulated category; see \cite[\S7]{Iyengar--Krause:2021}. 
	Associated to each $\fp$ in $\Spec R$ one has a local cohomology functor on $\dcat A$, denoted $\varGamma_{\fp}(-)$, extending the classical local cohomology functors over $R$, introduced by Grothendieck.
	This functor can be realised as 
	\[
	\varGamma_{\fp}M \cong  M \lotimes_R \varGamma_{\fp}R
	\]
	for $M$ in $\dcat A$, 
	where $\varGamma_{\fp}R$ is given by application of the functor
	\[
	E \longmapsto  \bigcup_{n > 0} \{ x \in E \mid \fp^n \cdot x   = 0\} 
	\]
	to an injective resolution of $R_{\fp}$.
	Observe that in $\dcat R$ one has 
	\[
	\Loc(\varGamma_\fp R) = \Loc(k(\fp)) \subseteq \Loc (R_\fp)\,.
	\]
	The equality can be deduced from arguments by Neeman \cite[Lemma~2.9 and proof of claim~(1) on page~528]{Neeman:1992a}, whilst the inclusion holds because $k(\fp)$ is an $R_\fp$-module. Applying the functor
	\[
	M\lotimes_R - \colon \dcat R\lra \dcat A
	\]
	yields the following:
	\[
	\Loc(\varGamma_\fp M) = \Loc(\fibre M\fp) \subseteq \Loc (M_\fp)\,.
	\]
	Our hypothesis is that $\fibre M\fp$ is in $\Loc(\fibre N{\fp})$ for each $\fp$. By the equalities above, this then yields that
	\[
	\Loc(\varGamma_\fp M) = \Loc (\fibre M\fp)   \subseteq  \Loc(\fibre N{\fp}) \subseteq  \Loc(N_\fp) 
	\]
	for each $\fp\in\Spec R$. Now we invoke the local to global principle to conclude that $M$ is in $\Loc(N)$ as desired.  When the Krull dimension of $R$ is finite the local to global principle is in Theorem~3.1 and Corollary~3.5 of~\cite{Benson--Iyengar--Krause:2011a}; the general case is due to Stevenson~\cite[Theorem~6.9]{Stevenson:2013}.
\end{proof}

\begin{corollary}
	\label{co:Hopkins}
	Let $A$ be a finite projective $R$-algebra, and let $M,N$ be $A$-complexes. When $M$ and $N$ are perfect, the following conditions are equivalent. 
	\begin{enumerate}
		\item \label{co:Hopkins1}
		$M\in  \Thick(N)$ in $\dcat A$;
		\item \label{co:Hopkins2}
		$\fibre M\fp\in \Thick(\fibre N{\fp})$ in $\dcat {\fibre A\fp}$ for each 
		$\fp \in \supp_RA$.
	\end{enumerate}
\end{corollary}
\begin{proof}
	As in the proof of Theorem~\ref{th:Neeman}, the implication \eqref{co:Hopkins1}$\Rightarrow$\eqref{co:Hopkins2} is clear. As to the converse, since $\Thick(\fibre N{\fp})\subseteq \Loc(\fibre N{\fp})$, the hypothesis and Theorem~\ref{th:Neeman} yield that $M$ is in $\Loc(N)$. Since $M$ and $N$ are perfect and hence compact in $\dcat A$, it follows that $M$ is in $\Thick(N)$, as claimed; see \cite[Proposition~3.4.13]{Krause:2021}.
\end{proof}

As is explained in the next paragraph, the results above extend results of Hopkins and Neeman dealing with commutative noetherian rings.

\begin{chunk}
	\label{ch:classical-HN}
	Set $A\colonequals R$. Then $\fibre R{\fp}= k(\fp)$, the residue field of $R$ at $\fp$, and one has $\fibre M{\fp}= M \lotimes_R k(\fp)$ for any $R$-complex $M$. The \emph{support} of $M$, denoted $\supp_RM$, is the collection $\fp\in\Spec R$ such that $H^*(\fibre M{\fp})\ne 0$. This coincides with the support of the $R$-module $H^*(M)$ when the latter is finitely generated; see \cite[\S2]{Foxby:1979}.
	
	Given $R$-complexes $M,N$, the condition that $\fibre M{\fp}$ is in $\Loc(\fibre N{\fp})$ 
	for each $\fp$ in $\supp_RA$
	is equivalent to: for each $\fp$ in $\supp_RA$ with $H^*(\fibre M{\fp})\ne 0$ it follows that $H^*(\fibre N{\fp})\ne 0$; equivalently, that $\supp_R M\subseteq \supp_RN$. So Theorem~\ref{th:Neeman} translates to 
	\[
	M\in \Loc(N) \text{ in $\dcat R$}\quad \iff\quad  \supp_RM\subseteq \supp_RN\,.
	\]
	This result is due to Neeman and is the key step in the classification
	of localising subcategories of $\dcat R$ via subsets of $\Spec
	R$; see ~\cite[Theorem~2.8]{Neeman:1992a}.
	
	In the same vein, Corollary~\ref{co:Hopkins} translates to: When $M,N$ are perfect complexes, $\supp_RM\subseteq \supp_RN$ if and only if $M$ is in $\Thick(N)$, which leads to a classification of the thick subcategories of the perfect $R$-complexes, due to Hopkins~\cite[\S4]{Hopkins:1987}.
\end{chunk}

Here is one simple application of the preceding results.

\begin{chunk}
	\label{ch:quivers}
	Let $Q$ be a finite quiver without oriented cycles. We consider the \emph{path algebra} $RQ$ which is, by the assumption on $Q$, a finite
	projective, even free, $R$-algebra. A basis of $RQ$ as an $R$-module is given by all paths in $Q$, and the multiplication is induced by the composition
	of paths. Modules over the path algebra $RQ$ identify with $R$-linear representations of the quiver $Q$.
	
	When $k$ is a field, $kQ$ has finite global dimension, and there is a well developed theory of quiver representations. In particular, thick subcategories of $\dbcat {kQ}$ have been classified in terms of non-crossing partitions \cite{Hubery--Krause:2016,Ingalls--Thomas:2009}. The quiver $Q$ determines a Coxeter group $W=W(Q)$ and a Coxeter element $c\in W$. We denote by $\mathrm{NC}(W,c)$ the corresponding poset of \emph{non-crossing  partitions}.  When the quiver $Q$ is of Dynkin type, there is a well defined assignment
	\[
	\dbcat {kQ}\supseteq \mathcal C\longmapsto \mathrm{cox}(\mathcal  C)\in\mathrm{NC}(W,c) 
	\]
	providing an isomorphism between the lattice of thick subcategories of $\dbcat {kQ}$ and the poset $\mathrm{NC}(W,c)$. 
	
	Let  $R$ be a commutative noetherian ring that is regular. For each $\fp$ in $\Spec R$ the ring $R_\fp Q$ has finite global dimension, because $R_\fp$ does.  Thus Theorem~\ref{th:bass-murthy} yields that each $M$ in $\dbcat{RQ}$ is perfect. Given this observation,  the fibre-wise criteria in Corollary~\ref{co:Hopkins} open a path to a classification of all thick subcategories of $\dbcat {RQ}$.
	Over a general commutative noetherian ring $R$, one gets an approach to classify the thick subcategories of perfect complexes in $\dbcat{RQ}$.
\end{chunk}

\section{Fibre-wise tilting complexes}
\label{se:tilting}
This section and the next are motivated by the question whether a derived equivalence of finite $R$-algebras amounts to a derived equivalence of their fibres over $\Spec R$. For this, we recall notions and results from derived Morita theory.

\subsection*{Derived equivalences}

Let $A$ be a finite $R$-algebra and $T$ a complex in $\dcat A$.
The ring 
\[
\End(T) \colonequals \op{\End_{\dcat A}(T)}
\]
admits a natural structure of a finite $R$-algebra.
The complex $T$ is \emph{tilting} if
\begin{enumerate}
	\item $\Ext^i_A(T,T)=0$ for $i\ne 0$, and
	\item
	$\Thick(T)=\Thick(A)$ in $\dcat A$.
\end{enumerate}
For such a complex  the functor 
\[
F 
\colonequals \RHom_A(T,-)
\colon \dbcat A \longiso \dbcat {\End(T)}
\]
is an equivalence of triangulated categories.
Vice versa, if $B$ is a noetherian ring which is derived equivalent to $A$, denoted by $A \sim B$ in the following, then there is an isomorphism of rings $B \cong \End(T)$ for a tilting complex $T$ of $A$. These statements follow from results of  Rickard~\cite{Rickard:1989, Rickard:1991b}; see also  Keller \cite{Keller:1994}.

The class of finite projective $R$-algebras is not closed under derived equivalences. This phenomenon was  studied by K\"onig and Zimmermann in \cite{Koenig--Zimmermann:1996}. It is related to tilting complexes that are not fibre-wise tilting.

\begin{example}\label{example:hereditary}
	Let $R$ be the polynomial ring $k[c]$ in one indeterminate over a field $k$ and $\fm$ the maximal ideal $(c)$ of $R$.
	The $k$-linear path algebra $A$ of the two-cycle quiver
	is a finite free $R$-algebra:
	\begin{align}
		\label{eq:two-cycle}
		{Q}\colon \begin{tikzcd}[ampersand replacement=\&] 1 \ar[yshift=3pt]{r}{{\alpha}} \& \ar[yshift=-3pt]{l}{{\beta}} 2 \end{tikzcd}
		\qquad
		A \colonequals k {Q} \cong 
		\begin{pmatrix} R & \fm \\ R & R \end{pmatrix}
		\subset \Mat_{2\times 2}(R)\,.
	\end{align}
	Replacing the projective summand $A e_2$ of $A$ with the simple module $A e_1 /A \alpha$ yields a tilting module of $A$ which has the projective resolution
	\[
	\begin{tikzcd}
		T \colonequals 
		(\,A e_2 \ar{r}{\cdot\alpha} & A\,)\,.
	\end{tikzcd}
	\]
	Since
	$A_{k(\fm)} \cong k {Q}/(\alpha\beta,\beta\alpha)$, it follows that
	$\Ext^{-1}_{A_{k(\fm)}}(T_{k(\fm)},T_{k(\fm)})\neq 0$, and thus the complex $T_{k(\fm)}$ is not tilting.
	On the other hand, the 
	quiver of the $R$-algebra $B \colonequals \End(T)$ has a sink:
	\[
	B \cong \begin{pmatrix} R & 0 \\ R/\fm& R/\fm \end{pmatrix} \cong k Q^*/I^* \quad \text{ for } \quad
	(Q^*,I^*)\colon
	\begin{tikzcd}
		1 \arrow[looseness=6,out=150,in=-150,swap,"x",
		""{name=X, inner sep=0, near end}] \arrow[r,"y",
		""{name=Y, inner sep=0, near start}
		]&2
	\end{tikzcd}
	\quad
	y x=0\,.
	\]
	Since $Z(B)\cong R$ is a Dedekind domain and $y \in B$ is a nonzero torsion element, 
	\ref{ch:subring} implies that $B$ cannot admit the structure of a finite projective $R$-algebra.
\end{example}
The last example motivates the following terminology.
\begin{definition}
	\label{de:fibre-wise-tilting}
	Let $A$ be a finite projective $R$-algebra. A complex $T$ in $\dcat A$ is called \emph{fibre-wise tilting} if the $\fibre A\fp$-complex $T_{k(\fp)}$ is tilting for each $\fp$ in $\Spec R$.
\end{definition}

In \cite[Theorem 2.1]{Rickard:1991b} Rickard proved a result on derived equivalences under base change  which implies the following.
\begin{theorem}
	\label{theorem:Rickard1}
	Let $A$ and $B$ be finite $R$-algebras  and $R \to S$ be a ring morphism of commutative noetherian rings such that $A \sim B$ and 
	\[ 
	\Tor_{n}^R(A, S) = 0 = \Tor_{n}^R(B,S)  \qquad\text{for  $n \geq 1$.}
	\]
	Then $ A \otimes_R S \sim B  \otimes_R S$. 
	More precisely, if $T$ is a tilting complex of $A$ such that $\End(T) \cong B$, then $T  \lotimes_R S$ is a tilting complex of $A \otimes_R S$ such that there is an isomorphism $\End(T \lotimes_R S) \cong B \otimes_R S$ of $S$-algebras.
	\qed
\end{theorem}

One idea to prove Theorem~\ref{theorem:Rickard1} is to compare the cohomology of the complex $X \colonequals \RHom(T,T)$ with that of $ X \lotimes_R \Gamma$.
This motivates the next considerations.
\begin{proposition}
	\label{pr:exact-lg}
	For $X\in \dbcat R$ the following conditions are equivalent.
	\begin{enumerate}
		\item \label{pr:exact-lg1}
		$H^i(X)=0$ for $i\ne 0$ and the $R$-module $H^0(X)$ is projective;
		\item \label{pr:exact-lg2}
		$H^i(\fibre X\fp)=0$ for $i\ne 0$ and for each $\fp \in \Spec R$;
		\item \label{pr:exact-lg3}
		$H^i(\fibre X\fm)=0$ for $i\ne 0$ and for each $\fm \in \max(\Spec R)$.
	\end{enumerate}
\end{proposition}
\begin{proof}
	\eqref{pr:exact-lg1}$\Rightarrow$\eqref{pr:exact-lg2} Since $H^i(X)=0$ for $i\ne 0$, one has that $X\cong H^0(X)$ in $\dcat R$. This gives, in $\dcat R$, the 
	first
	isomorphism below 	
	\[
	\fibre X\fp \cong \fibre {H^0(X)}\fp\cong H^0(X) \otimes_R k(\fp) \,.
	\]
	The second one holds because $H^0(X)$ is projective.
	Thus $H^i(\fibre X{\fp})=0$ for $i\ne 0$.
	
	\eqref{pr:exact-lg3}$\Rightarrow$\eqref{pr:exact-lg1} Since the $R$-modules $H^i(X)$ are finitely generated, it suffices to verify that for each maximal ideal $\fm$ the $R_\fm$-module $H^i(X)_\fm$ is zero for $i\ne 0$ and free when $i=0$. Replacing $R$ and $X$ by their localisation at $\fm$, we can thus assume $R$ is a local ring, with maximal ideal $\fm$. The hypothesis is that $H^i( X \lotimes_R k)=0$ for $i\ne 0$, where $k\colonequals R/\fm$. 
	
	Since $X$ is in $\dbcat R$ there is a quasi-isomorphism $F\iso X$ with each $F^i$ a finitely generated free $R$-module, zero for $i\gg 0$, and differential satisfying $d(F)\subseteq \fm F$; see \cite[Corollary~4.3.19]{Krause:2021}. In particular 
	\[
	H^i( X \lotimes_R k) \cong H^i(F \otimes_R k) = F^i \otimes_R k\,.
	\]
	Thus the hypothesis is equivalent to $F^i=0$ for $i\ne 0$, so that $X\cong F^0$ in $\dcat R$. This is the desired conclusion. 
\end{proof}

\subsection*{Intermediate fibres}
Let $\fp$ be a prime ideal of $R$.  In the next proposition, we call a commutative noetherian ring $S(\fp)$ an \emph{intermediate ring} if the natural ring morphism $R \to k(\fp)$ factors through $S(\fp)$. Classical examples are given by the local ring $R_{\fp}$ or its completion $\wh{R}_{\fp}$ at its maximal ideal.  For a finite projective $R$-algebra $A$ and a complex $T$ in $\dcat A$ we set
\[
A_{S(\fp)} \colonequals  A \otimes_R S(\fp)\quad\text{and}\quad T_{S(\fp)} \colonequals T   \lotimes_R S(\fp) \,.
\]
The algebra  $A_{S(\fp)}$ might be thought of as a thickened or intermediate fibre. Certain intermediate fibres will play a role in the lifting problems of the next section.  The following statement is the main result of the present one.

\begin{proposition}
	\label{pr:morita-v2}
	Let $A$ be a finite projective $R$-algebra and fix $T$ in $\dbcat A$. The following conditions are equivalent.
	\begin{enumerate}
		\item 
		\label{pr:morita1}
		The $A$-complex
		$T$ is  tilting and $\End(T)$ is projective as $R$-module;
		\item \label{pr:morita5}
		For each $\fp$ in $\Spec R$ and for any intermediate ring $S(\fp)$ the $A_{S(\fp)}$-complex 
		$T_{S(\fp)}$ is tilting and 
		$\End(T_{S(\fp)})$ is projective as $S(\fp)$-module;
		\item \label{pr:morita4}
		For each $\fp$ in $\Spec R$ there exists an intermediate ring $S(\fp)$ such that the $A_{S(\fp)}$-complex $T_{S(\fp)}$
		is tilting and 
		$\End(T_{S(\fp)})$ is projective as $S(\fp)$-module;
		\item \label{pr:morita2}
		For each $\fp$ in $\Spec R$ the $\fibre A{\fp}$-complex 
		$\fibre T{\fp}$ is tilting.
	\end{enumerate}
\end{proposition}

\begin{proof}
	The implications \eqref{pr:morita1}$\Rightarrow$\eqref{pr:morita5} and \eqref{pr:morita4}$\Rightarrow$\eqref{pr:morita2}
	follow from Theorem~\ref{theorem:Rickard1}.
	
	\eqref{pr:morita2}$\Rightarrow$\eqref{pr:morita1}  Assume that $T_{k(\fp)}$ is tilting for each $\fp$ in $\Spec R$. Then the $A$-complex $T$ is perfect
	by Proposition~\ref{pr:perfection}.  Set $X\colonequals \RHom_A(T,T)$, viewed as an $R$-complex. 
	Since $T$ is perfect, there is a natural isomorphism
	\begin{align}
		\label{eq:RHom}
		\fibre X{\fp} = \RHom_A(T,T) \lotimes_R k(\fp)   \cong \RHom_{\fibre A\fp}(\fibre T{\fp},\fibre T{\fp})\quad
	\end{align}
	in $\dcat R$ for each $\fp$ in $\Spec R$.  
	Given this, \eqref{pr:morita1} is readily deduced from Corollary~\ref{co:Hopkins} and Proposition~\ref{pr:exact-lg}.
\end{proof}

This proposition shows that the tilting property of a certain complex $T$ of $A$ with $R$-projective endomorphism ring is local with respect to any choice of intermediate fibres. In particular, fibre-wise tilting complexes  are precisely the tilting complexes with $R$-projective endomorphism rings.
The equivalence of \eqref{pr:morita1} and \eqref{pr:morita2}  explains the simultaneous absence of  these properties  in Example~\ref{example:hereditary}.
We will give another characterisation of fibre-wise tilting complexes  in \ref{ch:endo-proj}.

\begin{remark}
	The tilting property is local in a classical sense.  More precisely, for a finite $R$-algebra $A$ and complex $T$ in $\dbcat A$ the following are equivalent.
	\begin{enumerate}
		\item \label{pr:local-tilt1} The $A$-complex $T$ is tilting;
		\item \label{pr:local-tilt3} The $A_{\fp}$-complex  $T_{\fp}$ is tilting for each  $\fp$ in $\Spec R$.
	\end{enumerate}
	If $T$ is perfect, the above conditions are also equivalent to the following.
	\begin{enumerate}
		\setcounter{enumi}{2}
		\item \label{pr:local-tilt4} The $A_{\fm}$-complex  $T_{\fm}$ is tilting for each $\fm$ in $\max(\Spec R)$.
	\end{enumerate} 
	The proof of these equivalences uses a result due to Iyama and Kimura
	\cite[Theorem~2.8~(a)]{Iyama--Kimura:2021} and a characterisation of
	tilting complexes by Rickard \cite[Lemma~1.1]{Rickard:1991}.  The
	equivalence
	\eqref{pr:local-tilt1}$\Leftrightarrow$\eqref{pr:local-tilt3} holds
	for silting complexes as well
	\cite[Theorem~2.8~(b)]{Iyama--Kimura:2021}.
\end{remark}

\section{Lifting problems}
\label{se:lifting-problems}

Let $A$ be a finite $R$-algebra.
Since the residue field $k(\fp)$ of each local ring ${R}_{\fp}$ is isomorphic to that of its completion $\widehat{R}_{\fp}$, there are natural morphisms of rings
\[
\begin{tikzcd}
	A \ar[hookrightarrow]{r} & A_{\fp} = 	A 
	\otimes_R {R}_{\fp}
	\ar[hookrightarrow]{r} & \widehat{A}_{\fp} \colonequals
	A
	\otimes_R \widehat{R}_{\fp}
	\ar[twoheadrightarrow]{r}    
	& A_{k(\fp)} = A \otimes_R k(\fp) \,.
\end{tikzcd}
\]
Rickard's theorem~\ref{theorem:Rickard1}
yields a global-to-local picture on derived equivalences.
\begin{corollary}\label{corollary:global-to-local}
	
	Let $A$ and $B$ be finite projective $R$-algebras such that
	$A \sim B$.
	For each  
	$\fp$ in $\Spec R$ and   
	any intermediate ring $S(\fp)$ it follows that
	$A_{S(\fp)} \sim B_{S(\fp)}$.
	
	In particular,
	the following implications hold.
	\[
	\begin{array}{ccccccc}
		\text{{\footnotesize global}} && \text{{\footnotesize local}} &&  \text{{\footnotesize complete local}} && \text{{\footnotesize finite dimensional}} \\
		A \sim B & {\Longrightarrow} & 
		A_{\fp} \sim B_{\fp}  \ \forall \fp & {\Longrightarrow} &
		\widehat{A}_{\fp} \sim \widehat{B}_{\fp} \ \forall \fp & \overset{(\star)}{\Longrightarrow}
		& {A}_{k(\fp)}  \sim {B}_{k(\fp)} \ \forall \fp \,,
	\end{array}
	\] 
	where ``$\ \forall \fp$" abbreviates ``for any prime ideal $\fp$ of the ring $R$".\qed
\end{corollary}

In the present section, we focus mainly on the converse of the third implication,
which might be viewed as the problem to lift derived equivalences of algebras.
This problem can be solved in a certain setup where $R$ is a Dedekind domain and the fibres $A_{k(\fp)}$ and $B_{k(\fp)}$
are replaced by thickened versions (Theorem~\ref{theorem:lift-equivalences}).

\subsection*{Lifting tilting complexes and isomorphisms of algebras}
In \cite[§3]{Rickard:1991} Rickard proved lifting results for tilting complexes.
\begin{theorem}\label{theorem:Rickard2}
	Let $\Lambda$ be a finite free $S$-algebra over a complete local ring $S$ with residue field $k$ and let $T$ be a tilting complex of $\Lambda  \otimes_{S} k$.
	Then there is precisely one complex $L$ in $\dcat \Lambda$ up to isomorphism such that $$
	L \lotimes_{S} k
	\cong T \quad 
	\text{in}\quad\dcat {(\Lambda  \otimes_S k)}\,.
	$$
	Moreover, the complex $L$ is tilting, $\End(L)$ is a finite free $S$-algebra and there is an isomorphism of $k$-algebras $\End(L) \otimes_S k  \cong \End(T)$.\qed
\end{theorem}
This theorem yields a partial converse of $(\star)$ in Corollary~\ref{corollary:global-to-local}:
if $A_{k(\fp)}\sim B_{k(\fp)}$, then there is a finite free $\widehat{R}_{\fp}$-algebra $\Gamma$ with $\widehat{A}_{\fp} \sim \Gamma$ and $\Gamma \otimes_{\widehat{R}_{\fp}} k(\fp) \cong {B_{k(\fp)}}$.
So the problem to lift derived equivalences is reduced to a lifting problem for isomorphisms of algebras.
However, there are non-isomorphic commutative local noetherian rings with isomorphic residue fields.
The converse of $(\star)$ fails also for certain blocks of group algebras \cite[p.192]{Rickard:1998}.
Here is another quiver-theoretic example.

\begin{example}\label{example:skew-group-algebras}
	Let $n \geq 1$, and
	$A$ and $B$ denote the following matrix rings:
	\[
	A  \colonequals \begin{pmatrix} k[c] & (c) \\ k[c] & k[c] \end{pmatrix} 
	\supset
	B \colonequals \begin{pmatrix} k[c] & (c^n) \\ (c^n) & k[c] \end{pmatrix} \,.
	\]
	Similar to $A$, the subring $B$
	is also a finite free $R$-algebra whose centre is isomorphic to $R\colonequals k[c]$.
	In contrast to the quiver of $A$, the quiver of $B$ has relations:
	\[
	B \cong kQ/I \quad
	\text{for}\quad
	(Q,I)\colon 
	\begin{tikzcd} 
		\ar[looseness=6, out=-150, in=-210, "x"] 1 \ar[yshift=3pt]{r}{\alpha} & \ar[yshift=-3pt]{l}{\beta} 2  \ar[looseness=6, out=30, in=-30, "y"]
	\end{tikzcd} 
	\quad
	\begin{array}{cc}
		x^{2n} = \beta \alpha\,, & \alpha x = y \alpha\,, \\  
		y^{2n} = \alpha \beta\,, & \beta y = x \beta\,.	\end{array}
	\]
	The fibres
	$A_{k(\fp)}$ and $B_{k(\fp)}$ at 
	a prime ideal $\fp$ of $R$ are related as follows.
	\begin{enumerate}
		\item 
		If $c \notin \fp$, it holds that $(c)_{\fp} = (c^n)_{\fp} = R_{\fp}$ and there are isomorphisms 
		\[
		A_{\fp} \cong \Mat_{2 \times 2}(R_{\fp}) \cong B_{\fp} 
		\quad\text{and}\quad
		A_{k(\fp)} \cong \Mat_{2 \times 2}(R_{k(\fp)})\cong B_{k(\fp)}
		\] 
		of $R_{\fp}$-algebras and $k(\fp)$-algebras, respectively.
		\item If $c \in \fp$, it holds that $\fp = (c) \equalscolon \fm$. 
		Since $R/\fm \cong k(\fm)$, there are isomorphisms of $k(\fm)$-algebras
		\[
		A_{k(\fm)} \cong 
		\begin{pmatrix}
			k[c]/(c) & (c)/(c^2)\\
			k[c]/(c) & k[c]/(c)
		\end{pmatrix}
		\cong
		\begin{pmatrix}
			k[c]/(c) & (c^n)/(c^{n+1}) \\
			(c^n)/(c^{n+1}) & k[c]/(c)
		\end{pmatrix}
		\cong 
		B_{k(\fm)}\,.
		\]
		In other terms, both algebras $A_{k(\fm)}$ and $B_{k(\fm)}$ are isomorphic to the path algebras of the two-cycle quiver from \eqref{eq:two-cycle} with relations $\alpha \beta  = \beta \alpha= 0$.
		
		Since $\Lambda \colonequals \widehat{A}_{\fm}$ is a finite $\widehat{R}_{\fm}$-algebra over the complete local ring $\widehat{R}_{\fm}$, it is semiperfect,
		and thus any simple $\Lambda$-module is a direct summand of $\Lambda/ \rad {\Lambda}$,
		see \cite[Example~23.3 and Theorem~25.3]{Lam:1991}. 
		It follows that 
		\[
		\gldim \Lambda = \pdim_{\Lambda} \Lambda/{\rad \Lambda} = 1
		\]
		where the first equality holds according to  \cite[Proof of Proposition~2.2]{Iyama--Reiten:2008} and the second follows by computation.
		On the other hand, the projective dimension of any simple $\widehat{B}_{\fm}$-module, and thus the global dimension of $\widehat{B}_{\fm}$ are both infinite. 
		So $\widehat{A}_{\fm}$ and $\widehat{B}_{\fm}$ are not derived equivalent.
	\end{enumerate}	
	Together with Corollary~\ref{corollary:global-to-local}, this shows that
	despite the fact that 
	the fibres $A_{k(\fp)}$ and $B_{k(\fp)}$ are isomorphic at each ideal $\fp \in \mathrm{Spec} \,R$, the algebras $A$ and $B$ are not derived equivalent.
	Varying the parameter $n \geq 1$ of the algebra $B$ yields a family 
	of algebras, each of which is fibre-wise but not globally derived equivalent to $A$.
\end{example}

\begin{remark}
	In \cite[Lemma 8.1]{Eisele:2021} Eisele describes sufficient conditions
	($\Lambda$1)--($\Lambda$5) to lift the isomorphism of 
	$k(\fm)$-algebras $A_{k(\fm)} \cong B_{k(\fm)}$ to an isomorphism $\widehat{A}_{\fm} \cong \widehat{B}_{\fm}$
	of $\widehat{R}_{\fm}$-algebras.
	In Example~\ref{example:skew-group-algebras}, property ($\Lambda$3) fails. 
	Eisele has also studied the uniqueness of lifts in connection with derived equivalences \cite{Eisele:2012}.
	
\end{remark}
In \cite{Higman:1959}, 
Higman proved that it is possible to lift isomorphisms
of suitably large quotients in the following setup.
\begin{theorem}\label{th:Higman}
	Let $\Lambda$ and $\Gamma$ be finite free $S$-algebras over a complete discrete valuation ring $S$
	such that the $Q(S)$-algebra $\Lambda \otimes_S Q(S)$ is separable.
	Let $\fm$ denote the maximal ideal of $S$ and set
	$S_n \colonequals S/\fm^{n}$ for any $n \geq 1$.
	Then there is an integer $ \ell \colonequals \ell({\Lambda,\fm}) 
	\geq 1$ such that any isomorphism
	$\Lambda \otimes_S S_{\ell} \iso \Gamma \otimes_S S_{\ell}$ of $S_{\ell}$-algebras lifts to an isomorphism $\Lambda \iso {\Gamma}$ of $S$-algebras.  	\qed
\end{theorem}

\subsection*{Lifting derived equivalences}
We return to our initial setup of finite projective $R$-algebras $A$ and $B$. 
Instead of the fibre $A_{k(\fp)}$ at a prime ideal $\fp$ of $R$ we
consider possibly bigger intermediate quotients
\[
\begin{tikzcd}
	\widehat{A}_{\fp} =  A \otimes_R \widehat{R}_{\fp} \ar[twoheadrightarrow]{r} & 
	A_{\fp, n}
	\colonequals 
	\widehat{A}_{\fp} \otimes_{\widehat{R}_{\fp}} 
	\widehat{R}_{\fp}/ \fp^n  \widehat{R}_{\fp}  \ar[twoheadrightarrow]{r} & 
	A_{\fp,1} \cong A_{k(\fp)}
\end{tikzcd}
\]
for certain integers $n \geq 1$.
Rickard's theorems can be combined with Higman's result in order to lift derived equivalences of certain intermediate fibres.

\begin{theorem}\label{theorem:lift-equivalences}
	Let $\fp$ be a nonzero prime ideal of $R$, and $A$ and $B$ be finite projective $R$-algebras over a Dedekind domain $R$ such that the $Q(S)$-algebra $B \otimes_S Q(S)$
	is separable, where $S \colonequals \widehat{R}_{\fp}$.
	Let $\ell$ denote the number $\ell(\widehat{B}_{\fp},\fp S)$ 
	provided by Theorem~\ref{th:Higman}.
	Then $\widehat{A}_{\fp} \sim \widehat{B}_{\fp}$ if and only if there is an integer $n \geq \ell$ such that $A_{\fp,n} \sim B_{\fp, n}$.
\end{theorem}

\begin{proof}
	Let $n \geq 1$. The `only if'-implication follows from Theorem~\ref{theorem:Rickard1}.
	
	To show the converse, set 
	$\Lambda \colonequals \widehat{A}_{\fp}$,  $S_n \colonequals S/\fp^n S$ and $\Lambda_n \colonequals \Lambda \otimes_S S_n$.
	Let $T$ be a tilting complex of $\Lambda_n = A_{\fp,n}$
	with $\End(T) \cong B_{\fp, n}$.
	
	Theorem~\ref{theorem:Rickard1} implies that $T_1 \colonequals T \lotimes_{S_n} S_1$ is a tilting complex of $\Lambda_1$ such that $\End(T_{1})\cong B_{\fp,1}$.
	Because of Theorem~\ref{theorem:Rickard2}, $T_1$ lifts to a tilting complex $L$ of $\Lambda$ 
	with $S$-free endomorphism ring $\Gamma$.
	Theorem~\ref{theorem:Rickard1} yields that $L_n \colonequals  L \lotimes_S S_n $ is a tilting complex of $\Lambda_n$ with $\End(L_n) \cong \Gamma  \otimes_S S_n \equalscolon \Gamma_n$.
	Since $L_n$ and $T$ are both lifts of $T_1$ to $\dcat {\Lambda_n}$,
	it follows that $L_n \cong T$
	by the uniqueness statement in Theorem~\ref{theorem:Rickard2}. Thus, there are isomorphisms of $S_n$-algebras 
	$\Gamma_n \cong \End(L_n) \cong \End(T) \cong B_{\fp,n}$.
	
	Applying the argument above to an integer  $n \geq \ell$, Theorem~\ref{th:Higman} yields an isomorphism $\Gamma \cong \widehat{B}_{\fp}$ of $S$-algebras. 
	This shows that $\widehat{A}_{\fp} = \Lambda \sim\Gamma \cong \widehat{B}_{\fp}$.
\end{proof}
The last statement can be viewed as a first step to reduce
derived equivalence problems of Gorenstein algebras
to finite dimensional algebras.
\begin{remark}
	Rickard's theorems and the last proof yield a bijection between tilting complexes of $\widehat{A}_{\fp}$ with $\widehat{R}_{\fp}$-free endomorphism rings and those of any intermediate quotient $A_{\fp,n}$.
	This bijection extends to certain settings without the $R$-projectivity assumption on $A$ as well as to silting complexes \cite{Gnedin:2022}, see also \cite{Eisele:2021}. 
\end{remark}

\section{Gorenstein algebras}
\label{se:Gorenstein-algebras}
An associative ring $A$ is \emph{Iwanaga--Gorenstein} if it is noetherian on both sides, and both $\injdim A$ and $\injdim \op A$ are finite. Zaks~\cite{Zaks:1969} proved that in this case $\injdim A=\injdim \op A$; see also \cite[Lemma~6.2.1]{Krause:2021}.
\begin{definition}
	\label{de:gor-algebra}
	Following \cite{Iyengar--Krause:2021} an associative ring $A$ is a \emph{Gorenstein $R$-algebra} if it satisfies the following conditions.
	\begin{enumerate}
		\item \label{Gorenstein-fp}
		$A$ is a finite projective $R$-algebra;
		\item \label{Gorenstein-ig}
		$A_\fp$ is Iwanaga--Gorenstein for each $\fp$ in $\supp_RA$.
	\end{enumerate}
\end{definition}
Here $\supp_RA$ denotes the support of $A$ as an $R$-module; see \ref{ch:support}.  It is immediate from the definition that $A$ is a Gorenstein $R$-algebra if and only if $\op A$ is a Gorenstein $R$-algebra.

\begin{example}
	Any commutative Gorenstein ring $R$ is Gorenstein as an $R$-algebra.  Any algebra that is finite dimensional over a field $k$ and Iwanaga--Gorenstein is clearly a Gorenstein $k$-algebra.
	
	Let $R$ be commutative Gorenstein ring. For any finite group $G$, the group algebra $RG$ is a Gorenstein $R$-algebra. This can be checked directly, or, better still, it follows from Theorem~\ref{th:gor-family}.  In the same vein, any finitely generated exterior algebra over $R$, or any matrix ring over $R$, is a Gorenstein $R$-algebra; see also Theorem~\ref{th:morita}.
\end{example}

In all the examples above we started with $R$ being Gorenstein. This is a necessity, because of the result below from \cite[Lemma~4.1]{Iyengar--Krause:2021}. Its converse need not hold: consider the case of an algebra finite dimensional over a field.

\begin{lemma}
	\label{le:descent}
	When $A$ is a Gorenstein $R$-algebra, the ring $R_\fp$ is Gorenstein for each $\fp$ in $\supp_RA$.
\end{lemma}

\begin{proof}
	Fix a $\fp$ in $\supp_RA$ and let $k(\fp)$ denote the residue field at $\fp$. One has 
	\[
	\Ext^i_{R_\fp}(k(\fp), A_\fp) \cong \Ext^i_{A_\fp}(A_\fp \otimes_{R_{\fp}} k(\fp), A_\fp) =0 \qquad \text{for $i\gg 0$\,,}
	\]
	where the first isomorphism holds by adjunction, and the second one holds because $A_\fp$ is Iwanaga--Gorenstein. Since the $R_\fp$-module $A_\fp$ is finite free and nonzero---this is where we use the fact $\fp$ is in $\supp_RA$---it follows that
	\[
	\Ext^i_{R_\fp}(k(\fp), R_\fp) =0 \qquad\text{for $i\gg 0$\,.}
	\]
	Thus the local ring $R_\fp$ is Gorenstein; see \ref{ch:Gor-local}. 
\end{proof}

The argument above only uses the hypothesis that $A_\fp$ has finite injective dimension on one side. This observation will be used in the proof of Theorem~\ref{th:cnrings}.  

\begin{chunk}\label{ch:subring2}
	As explained in \ref{ch:subring}, when $A$ is a finite projective $R$-algebra one can drop down to a ring factor of $R$ and ensure  $A$ is also faithful as an $R$-module, equivalently, that $\supp_RA=\Spec R$. In this case when $A$ is a Gorenstein $R$-algebra, the ring $R$ is Gorenstein, by Lemma~\ref{le:descent}. 
\end{chunk}

\subsection*{Dependence on base}
Next we discuss the dependence of the Gorenstein property on the ring $R$.  The following result, essentially contained in \cite[Theorem~4.6]{Iyengar--Krause:2021}, is the key. 

\begin{theorem}
	\label{th:goto-plus}
	Let $R$ be a commutative noetherian ring and $A$ a finite projective $R$-algebra. Then $A$ is Gorenstein as an $R$-algebra if and only if 
	\[
	\Ext^i_A(M,A)=0 = \Ext^i_{\op A}(N,A) \qquad\text{for $i\gg 0$}
	\]
	for each $M\in \mod A$ and $N\in \mod\op A$. 
\end{theorem}

\begin{proof}
	In view of the discussion in \ref{ch:subring2}, one can assume $A$ is also faithful as an $R$-module. Then the forward implication is part of \cite[Theorem~4.6]{Iyengar--Krause:2021}. By the same token, for the converse statement it suffices to prove that the stated vanishing of Ext modules implies that the ring $R$ is Gorenstein. The argument for this is basically in the proof of Lemma~\ref{le:descent}.
	
	Let $\fp$ be a prime ideal of $R$. Then for $i\gg 0$ one has
	\[
	\Ext^i_{R_\fp}(k(\fp),A_\fp)\cong \Ext^i_{R}(R/\fp,A)_\fp \cong \Ext^i_{A}(A \otimes_{R} R/\fp, A)_\fp =0 
	\]
	where the first isomorphism holds since localisation is flat and the second one is by adjunction. The equality is by hypothesis, which applies as $A\otimes_R R/\fp$ is in $\mod A$. Since $A$ is a finite projective $R$-module, $A_\fp$ is a finite free $R_\fp$-module; it is non-zero because $A$ is faithful over $R$. Thus $\Ext^i_{R_\fp}(k(\fp),R_\fp)=0$ for $i\gg 0$, so that $R_\fp$ is Gorenstein. Since this holds for any $\fp$ in $\Spec R$, the ring $R$ is Gorenstein.
\end{proof}

Since the vanishing of Ext modules in the statement above has nothing to do with the ring $R$, the result above has the following consequence.

\begin{corollary}
	\label{co:independence}
	Let $R$ and $S$ be commutative noetherian rings, and $A$ an associative ring that is finite projective over $R$ and over $S$. Then $A$ is Gorenstein as an $R$-algebra if and only if it is Gorenstein as an $S$-algebra. \qed
\end{corollary}

\subsection*{Dualising bimodule}
Let $A$ be a finite projective $R$-algebra. The $A$-bimodule 
\[
\omega_{A/R}\colonequals \Hom_R(A,R)
\]
is the \emph{dualising bimodule} of the $R$-algebra $A$.

The following characterisation of the Gorenstein property is contained in  \cite[Theorem~4.6]{Iyengar--Krause:2021}. Its proof uses the result of Bass and Murthy recalled in \ref{th:bass-murthy}. 

\begin{theorem}
	\label{th:dualising}
	Let $R$ be a commutative Gorenstein ring and  $A$ a finite projective $R$-algebra. Then $A$ is a Gorenstein $R$-algebra if and only if the $A$-bimodule $\omega_{A/R}$ is perfect on both sides. \qed
\end{theorem}

When $A$ is commutative, the forward direction of the preceding statement can be strengthened to: the $A$-module $\omega_{A/R}$ is projective. This is one of the special features of commutative algebras.

As explained in Section~\ref{se:finite-projective}, a finite projective $R$-algebra $A$ can be viewed as a family of finite dimensional algebras, parameterised by $\Spec R$. The result below lends further credence to this perspective.

\begin{theorem}
	\label{th:gor-family}
	Let $A$ be a finite projective $R$-algebra.   Then $A$ is a Gorenstein $R$-algebra if and only if  $R_\fp$ is Gorenstein and $\fibre A\fp $ is Iwanaga--Gorenstein for each $\fp\in\supp_RA$; equivalently, for each $\fm\in\max(\supp_RA)$. 
\end{theorem}

\begin{proof}
	Fix $\fp$ in $\supp_RA$. The hypothesis that $A$ is finite projective over $R$ yields
	\begin{align}\label{eq:dual/fibre}
		\quad 
		\fibre {(\omega_{A/R})} \fp =  \Hom_R(A,R)\lotimes_R k(\fp)  \iso \Hom_{k(\fp)}(\fibre A\fp, k(\fp)) = \omega_{\fibre A\fp/k(\fp)}
	\end{align}
	as $\fibre A\fp$-bimodules. The stated result is immediate from the characterisation of the Gorenstein property in terms of the dualising bimodule; see Theorem~\ref{th:dualising} and Proposition~\ref{pr:perfection}.
\end{proof}

\begin{corollary}
	If $A$ and $B$ are Gorenstein $R$-algebras, then so is $A\otimes_RB$. 
\end{corollary}

\begin{proof}
	Given Theorem~\ref{th:gor-family}, it suffices to verify that when $k$ is a field and $A$ and $B$ are finite dimensional $k$-algebra that are Iwanaga--Gorenstein, then so is $A\otimes_kB$. This is well-known; see, for instance, \cite[Proposition~6.2.6]{Krause:2021}.
\end{proof}

It is known that the dualising bimodule $\omega_{A/k}$ 
of a finite dimensional $k$-algebra $A$ is already projective on both sides if it is projective on
one side, see, for example,
\cite[Propositions~I.8.16~(i) and~II.7.6]{Skowronski--Yamagata:2011}.
This fact extends to finite projective $R$-algebras.
\begin{lemma} \label{le:dual-proj}
	Let $A$ be a finite projective $R$-algebra.
	Then $\omega_{A/R}$ is projective as an $A$-module if and only if $\omega_{A/R}$ is projective as an $\op A$-module.
\end{lemma}
\begin{proof}
	Because of the $R$-algebra isomorphism $\op{(\op A)} \cong A$, it suffices to show the `only if'-implication. Assume that $M \colonequals \omega_{A/R}$ is projective as an $A$-module.
	
	Let $\fp$ be a prime ideal of $R$.
	As $A$ is finite projective as $R$-module, so is $M$.
	Using Lemma~\ref{le:proj} and that the isomorphism in \eqref{eq:dual/fibre} is $\fibre A \fp$-linear, it follows that the $\fibre A \fp$-module $N \colonequals \omega_{\fibre A \fp/ k(\fp)}$ is projective.
	Since $\fibre A \fp$ is a finite dimensional $k(\fp)$-algebra, $N$ is also projective as 
	$\op {(A_{k(\fp)})}$-module.
	Using that the isomorphism in \eqref{eq:dual/fibre} is also $\op {(\fibre A \fp)}$-linear together with the isomorphism $\op {( \fibre A \fp )} \cong {\fibre {(\op A)} \fp}$ of $k(\fp)$-algebras,
	it holds that $\fibre {M} \fp$ is projective as $\fibre {(\op A)} \fp$-module for any $\fp \in \Spec R$.
	
	As the $R$-algebra $\op A$ and the $\op A$-module $M$ are both finite projective over $R$, another application of  Lemma~\ref{le:proj} implies that $M$ is projective over $\op A$.
\end{proof}

\subsection*{Derived Morita invariance}
Here is our main result on the Morita invariance of the Gorenstein property. Fibre-wise tilting objects are introduced in Definition~\ref{de:fibre-wise-tilting}.

\begin{theorem}
	\label{th:morita}
	Let $A$ and $B$ be finite projective $R$-algebras which are derived equivalent.
	Then $A$ is a Gorenstein $R$-algebra if and only if $B$ is a Gorenstein $R$-algebra.
\end{theorem}

\begin{proof}
	Assume that the $R$-algebra $A$ is Gorenstein.  
	Let $T$ be a tilting complex with $\End(T) \cong B$.
	Consider the equivalence $$F\colonequals \RHom_A(T,-) \colon \dbcat A \iso \dbcat B.$$ Fix $M$ in $\dbcat A$. Since $T$ and $A$ generate the same thick subcategory in $\dcat A$ one gets the first equivalence below:
	\begin{align*}
		\Ext^i_A(M,A) =0  \quad\text{for $i\gg 0$} 
		&\iff \Ext^i_A(M,T) \quad\text{for $i\gg 0$} \\
		&\iff \Ext^i_B(F(M),B) \quad\text{for $i\gg 0$} 
	\end{align*}
	The second one holds because $F(T)\cong B$ and $F$ is an equivalence of categories. Since $F$ is an equivalence, we conclude that 
	$\Ext^*_A(M,A)$ is bounded for each $M$ in $\dbcat A$ if and only if $\Ext^*_B(N,B)$ is bounded for each $N$ in $\dbcat B$.
	
	Since $T$ is a tilting complex in $\dcat A$, its dual $\RHom_A(T,A)$ is tilting in $\dcat{\op A}$, with endomorphism algebra $\op B$; see \cite[Remark~9.2.7]{Krause:2021}.  Thus the conclusion above applies also to $\dbcat{\op A}$ and $\dbcat{\op B}$.
	
	Theorem~\ref{th:goto-plus} now yields that 
	$B$ is Gorenstein. Interchanging $A$ and $B$, the converse follows by the same arguments.
\end{proof}

A (classical) Morita equivalence is fibre-wise tilting, so  Proposition~\ref{pr:morita-v2} and the preceding result yield:
\begin{corollary}
	\label{co:Gor-Morita}
	Let $A$ be a Gorenstein $R$-algebra. If $B$ is Morita equivalent to $A$, then $B$ is a Gorenstein $R$-algebra as well. \qed
\end{corollary}

\begin{chunk}\label{ch:endo-proj}
	Let $A$ be a finite $R$-projective algebra and $T$ a tilting complex of $A$.
	By \cite[Corollary 2.3, Lemma 4.1 (3)]{Abe--Hoshino:2007}
	the $R$-algebra $\End(T)$ is projective as $R$-module if and only if 
	$\Ext^i_A(T, \omega_{A/R} \lotimes_A T) = 0$ for any $i > 0$.
	In particular, $T$ is fibre-wise tilting if
	$T \cong \omega_{A/R} \lotimes_A T$ in $\dcat A$.  
\end{chunk}

\begin{chunk}
	A finite $R$-algebra $A$ is \emph{symmetric}
	if there is an isomorphism $\omega_{A/R} \cong A$ of $A$-bimodules.
	If $R$ is Gorenstein, a finite projective symmetric $R$-algebra is Gorenstein by Theorem~\ref{th:dualising}.  	
	
	By a Theorem of Abe and Hoshino \cite[Theorem 4.3]{Abe--Hoshino:2007}, symmetric finite projective $R$-algebras, and thus symmetric Gorenstein $R$-algebras because of \ref{th:morita}, are closed under derived equivalences.
	The same is true for symmetric finite $R$-algebras over a Gorenstein ring $R$ satisfying certain depth conditions \cite[Theorems 3.3, 3.1(1)]{Iyama--Reiten:2008}.
\end{chunk}

\subsection*{Cohomologically noetherian rings}
Let $A$ be a noetherian ring. As in \cite[\S4]{Benson--Iyengar--Krause:2008a}, the graded-centre of $\dbcat A$ is denoted $Z^*(\dbcat A)$. Following Avramov and Iyengar, we say the ring $A$ is \emph{cohomologically noetherian} if there exists a graded-commutative ring $S$ and a homomorphism of graded rings $S\to Z^*(\dbcat A)$ such that $\Ext^*_A(M,N)$ is noetherian as an $S$-module for all $M,N$ in $\dbcat A$.  In particular $A$ is noetherian as $S^0$-module, and hence $Z(A)$ is a noetherian ring, and $A$ is finitely generated as a $Z(A)$-module; this follows, for example, from work of Formanek \cite{Formanek:1973}.

\begin{theorem}
	\label{th:cnrings}
	Let $A$ be a finite projective $R$-algebra. If $A$ is cohomologically noetherian (with respect to  $S$), then so is $\op A$ (also with respect to  $S$), and the $R$-algebra $A$ is Gorenstein.
\end{theorem}

\begin{proof}
	We can assume that $R$ is a subring of the centre of $A$; see the discussion in \ref{ch:subring}. For any $M$ in $\dbcat A$ the action of  $S$ on $\Ext^*_A(M,A)$ factors through $\Ext^*_A(A,A)=A$, so the  cohomologically noetherian condition implies 
	\[
	\Ext_A^i(M,A)=0 \qquad \text{for $i\gg 0$\,.}
	\]
	Now an argument akin to the one in proof of Lemma~\ref{le:descent} yields that the ring $R$ is Gorenstein. Here are details: We wish to verify that $R_\fm$ is Gorenstein for each maximal ideal $\fm$ of $R$. With $k$ denoting the residue field at $\fm$, from the preceding discussion, for $i\gg 0$ one gets the equality below
	\[
	\Ext^i_{R_\fm}(k,A_\fm) \cong \Ext^i_R(k,A) \cong  \Ext^i_A(A \otimes_R k,A) = 0\,.
	\]
	The first isomorphism holds because $\Ext^i_R(k,A)$ is annihilated by $\fm$ and so it is already $\fm$-local. Since $A_\fm$ is a non-zero free $R_\fm$-module of finite rank, it follows that $\Ext^i_{R_\fm}(k,R_\fm)=0$ for $i\gg 0$, and hence $R_\fm$ is Gorenstein; see \ref{ch:injdim-test}.
	
	Given that $R$ is Gorenstein, Theorem~\ref{th:goto+} yields that the functor 
	\[
	\RHom_R(-,R)\colon \dbcat A\lra \dbcat{\op A}
	\]
	is a (contravariant) equivalence of categories; it is its own quasi-inverse. It follows from this that $\op A$ is cohomologically noetherian with respect to  $S$, and hence as before that $\Ext^i_{\op A}(N,A)=0$ for each $N$ in $\dbcat{\op A}$ and $i\gg 0$.  
	
	Now Theorem~\ref{th:goto-plus} yields that the $R$-algebra $A$ is Gorenstein.
\end{proof}

\begin{example}
	Examples of rings that are cohomologically noetherian include group algebras $kG$, where $k$ is a field and $G$ is a finite group or even 
	a finite (super) group scheme. In these cases the ring $S$ that witnesses this property is the cohomology ring of $G$. 
	
	More generally, any finite dimensional algebra $A$ over a field $k$ satisfying the Fg property, in the sense of Erdmann, Holloway, Snashall, Solberg, Taillefer~\cite{Erdmann--Holloway--Snashall--Solberg--Taillefer:2004}, is cohomologically noetherian, with $S$ the Hochschild cohomology ring of $A$ over $k$.
\end{example}

\subsection*{The work of Buchweitz}  Let $R$ be a commutative noetherian ring. In Chapter 7 of ~\cite{Buchweitz:1986} Buchweitz studies maps of rings $\varphi\colon R\to A$ satisfying the following properties.
\begin{enumerate}
	\item \label{eq:fp}
	$A$ is a finite $R$-algebra, of finite projective dimension as an $R$-module;
	\item \label{eq:IG}
	$A$ is Iwanaga--Gorenstein ring.
\end{enumerate}
Condition \eqref{eq:fp} is weaker than the corresponding condition defining a Gorenstein algebra~\ref{de:gor-algebra}, whereas \eqref{eq:IG} is stronger. It seems plausible to develop a theory of Gorenstein algebras along the lines presented above and in \cite{Iyengar--Krause:2021},  where, instead of requiring that $A$ is a projective $R$-algebra, it is allowed to be of finite projective dimension over $R$, as in Buchweitz's definition.  In that context, the fibres $\fibre A{\fp}$ would no longer be rings, but rather differential graded algebras,  finite dimensional over the field $k(\fp)$.  One would have to contend with not-necessarily-commutative Gorenstein dg algebras, which have been studied by Frankild and J{\o}rgensen~\cite{Frankild--Jorgensen:2003}, and also by Dwyer, Greenlees, and Iyengar~\cite{Dwyer--Greenlees--Iyengar:2005}. One motivation for pursuing such a project is that the proposed definition captures some natural families of examples that ought to be considered Gorenstein but are not covered by our present framework. Gentle algebras, discussed in the next section, are one such, as are algebras which are derived equivalent to Gorenstein algebras via complexes that are tilting but not necessarily fibre-wise tilting.

\section{Maximal Cohen--Macaulay modules}
\label{se:maximal-cohen-macaulay-modules}

In this section we discuss maximal Cohen--Macaulay modules over Gorenstein algebras. Once again, we begin with commutative rings.

\begin{chunk}
	Let $(R,\fm,k)$ be a commutative local ring. A finitely generated $R$-module $M$ is \emph{maximal Cohen--Macaulay} if $\depth_RM\ge \dim R$; equivalently, if
	\[
	\Ext^i_R(k,M)=0 \qquad\text{for $i<\dim R$\,.}
	\]
	When $M\ne 0$ one always has $\Ext^{\dim R}_R(k,M)\ne 0$, see \cite[Theorem~(1.1)]{Fossum--Foxby--Griffith--Reiten:1975}, so such a module is maximal Cohen--Macaulay if and only if $\depth_RM=\dim R$. 
	
	Let now $R$ be a Gorenstein local ring. It is immediate from the Grothendieck's local duality theorem~\cite[Theorem~3.5.8]{Bruns--Herzog:1998a} that $M$ is maximal Cohen--Macaulay if and only if 
	\[
	\Ext_R^i(M,R) = 0 \qquad \text{for $i\ge 1$\,.}
	\]
	Observe that since $R$ is Gorenstein one has $\injdim R=\dim R$, as explained in \ref{ch:injdim-test}, so for any $M\in \Mod R$ one has
	\[
	\Ext^i_R(M,R)=0\qquad \text{for $i\ge \dim R+1$\,.}
	\]
	Thus a simple dimension shifting argument shows that for any $M$ in
	$\mod R$ the $R$-modules $\syz^i_R(M)$ are maximal Cohen--Macaulay for $i\ge \dim R$.
\end{chunk}

This suggests the following definition.

\begin{chunk} 
	\label{chunk:mcmR} 
	Let $R$ be a commutative Gorenstein ring. An $R$-module $M$ in $\mod R$ is \emph{maximal Cohen--Macaulay} if $\Ext^i_R(M,R)=0$ for $i\ge 1$.
	
	Since a finitely generated $R$-module is zero if and only if its localisation at any prime (equivalently, maximal) ideal is zero, a finitely generated $R$-module $M$ is maximal Cohen--Macaulay if and only if the $R_\fp$-module $M_\fp$ is maximal Cohen--Macaulay of all $\fp$ in $\Spec R$; equivalently, for all $\fp$ in $\max(\Spec R)$.
	
	As in the local case one has the following result.

	\begin{corollary}
		\label{co:goto}
		Let $R$ be a commutative Gorenstein ring and $M$ a finitely generated
		$R$-module and set $s\colonequals \sup\{i\in\bbN\mid \Ext^i_R(M,R)\neq 0\}$. The $R$-module
		$\syz^i_R(M)$ is maximal Cohen--Macaulay for $i\ge s$. \qed
	\end{corollary}
	
	Goto's theorem~\ref{th:goto} ensures that the $s$ is finite, so for any $M$ all high syzygy modules will be maximal Cohen--Macaulay. 
	The noteworthy aspect is that there is no uniform bound on $s$, since $\injdim R$ need not be finite; see \ref{ch:Nagata}.
\end{chunk}

\subsection*{Gorenstein algebras}
Let $A$ be a Gorenstein $R$-algebra. A finitely generated $A$-module $M$ is \emph{maximal Cohen--Macaulay} if $\Ext^i_A(M,A)=0$ for $i\ge 1$. Observe that this property does not depend on $R$; see also Corollary~\ref{co:independence}.

\begin{lemma}
	\label{le:mcm-base-independence}
	Let $A$ be a Gorenstein $R$-algebra and $M$ a finitely generated $A$-module. If $M$ is maximal Cohen--Macaulay as an $A$-module, then it is maximal Cohen--Macaulay also as an $R$-module. The converse holds when $\omega_{A/R}$ is projective over $A$ or over $\op A$, and, in particular, when $A$ is commutative. 
\end{lemma}

\begin{proof}
	When $M$ is maximal Cohen--Macaulay it follows by a standard dimension shifting argument that $\Ext^i_A(M,W)=0$ for $i\ge 1$ and any $A$-module $W$ of finite projective dimension. Applying this observation to $W\colonequals \omega_{A/R}$, and this is legitimate by Theorem~\ref{th:dualising}, justifies the equality below:
	\[
	\Ext_R^i(M,R) \cong \Ext_A^i(M,\omega_{A/R}) =0 \quad \text{for $i\ge 1$\,.}
	\]
	The isomorphism is by adjunction. Thus $M$ is maximal Cohen--Macaulay also as an $R$-module. The isomorphism also yields that when $M$ is maximal Cohen--Macaulay as an $R$-module, one has $\Ext_A^i(M,\omega_{A/R})=0$ for $i\ge 1$. 
	
	The two projectivity assumptions on $\omega_{A/R}$ are equivalent by Lemma~\ref{le:dual-proj}.
	When the $\op A$-module $\omega_{A/R}$ is projective,
	then $A \cong \Hom_R(\omega_{A/R}, R)$ is a direct summand of $\omega_{A/R}^n$ for a certain integer $n \geq 1$, which implies $\Ext_A^i(M,A)=0$ for $i\geq 1$, yielding the converse.
\end{proof}

In the statement above, the converse does not hold for general non-commutative rings: consider, for example, a finite dimensional algebra $A$ over a field $k$ such that $A$ is  Gorenstein but not self-injective. In this case not every finite dimensional $A$-module is maximal Cohen--Macaulay, but every $k$-module is maximal Cohen--Macaulay. This highlights an important feature of the commutative context.

Nevertheless many aspects of the theory of maximal Cohen--Macaulay over commutative rings carry over, including Goto's theorem~\ref{th:goto} and its consequence~\ref{th:goto+}. This is recorded in the statement below, which can be proved by a simple reduction to the commutative case; see \cite[Theorem~4.6]{Iyengar--Krause:2021} for details.

\begin{theorem}
	When $A$ is a Gorenstein algebra, the functor $\RHom_A(-,A)$ induces an equivalence of categories
	\[
	\RHom_A(-,A) \colon\op{ \dbcat A} \longiso \dbcat{\op A}\,.
	\]
	In particular, for any finitely generated $A$-module $M$ one has $\Ext_A^i(M,A)=0$ for $i\gg 0$, so the syzygy modules $\syz^i_A(M)$ are maximal Cohen--Macaulay for $i\gg 0$. \qed
\end{theorem}

The collection of maximal Cohen--Macaulay modules over a Gorenstein $R$-algebra $A$ form a Frobenius exact category, with exact structure inherited from $\mod A$, with projective/injective objects the finitely generated projective $A$-modules. We denote by $\smcm A$  the corresponding stable category.
For any $A$-modules $M$, $N$, we write $\sHom_A(M,N)$ for the quotient of $\Hom_A(M,N)$ modulo the ideal of morphisms which factor through a projective $A$-module.

Let $\dsing(A)$ be the \emph{singularity category} of $A$ introduced
by Buchweitz~\cite{Buchweitz:1986} and independently by
Orlov~\cite{Orlov:2004a}. This is the quotient of $\dbcat A$ by the
thick subcategory of perfect complexes; see
\cite[\S6.2]{Krause:2021}. The following result was proved by
Buchweitz~\cite[Theorem~4.4.1]{Buchweitz:1986} for Iwanaga--Gorenstein
rings. The version stated below is from
\cite[Theorem~6.5]{Iyengar--Krause:2021}.

\begin{theorem}
	\label{th:buchweitz}
	When $A$ is a Gorenstein $R$-algebra the embedding of the category of maximal Cohen--Macaulay modules into $\dbcat A$ induces an equivalence of triangulated categories $\smcm A\iso \dsing(A)$. \qed
\end{theorem}

As in Buchweitz's work~\cite{Buchweitz:1986} one can also identify $\smcm A$  with the homotopy category of acyclic complexes of
projective $A$-modules; see \cite[\S6]{Iyengar--Krause:2021}.

\subsection*{Duality}
We are interested in duality properties of $\smcm A$. To that end, we record that when $A$ is a Gorenstein $R$-algebra the exact functor
\[
\omega_{A/R}\lotimes_A-\colon \dcat A \longiso \dcat A
\]
is an equivalence, and induces an equivalence on $\dbcat A$; see \cite[Theorem~4.5]{Iyengar--Krause:2021}. Since the $A$-module $\omega_{A/R}$ is perfect on either side, the functor takes perfect complexes to perfect complexes and hence induces an equivalence of categories
\[
\omega_{A/R}\lotimes_A- \colon \dsing(A) \longiso \dsing(A)\,.
\]
See \cite[Proposition~6.7]{Iyengar--Krause:2021}.

\begin{chunk}\label{ch:isolated}
	Let $R$ be a commutative noetherian local ring and $E$ an injective hull of its residue field. We consider the Matlis duality functor
	\[
	(-)^\vee \colonequals \Hom_R(-,E)
	\]
	on the category of $R$-modules. A  finite $R$-algebra $A$ has an \emph{isolated singularity}
	if the ring $A_\fp$ has finite global dimension for each non-maximal prime $\fp$ in $\Spec R$. 
	This terminology is motivated by work of Buchweitz \cite[Definition~7.7.3]{Buchweitz:1986} and differs from Auslander's notion of isolated singularities.
	
	When a Gorenstein $R$-algebra $A$ has an isolated singularity, for any objects
	$M,N$ in $\smcm A$, 
	it holds that
	$\sHom_A(M,N)_{\fp} \cong \sHom_{A_{\fp}}(M_{\fp},N_{\fp}) = 0$
	for each non-maximal prime ideal $\fp$
	by \cite[Lemma~7.6~c)]{Auslander:1978a}, and together with 
	Theorem~\ref{th:buchweitz} it follows that
	each morphism space in $\dsing(A)$ is an $R$-module of finite length; see also \cite[Subsection~(C.4.2)]{Buchweitz:1986}.
	\begin{theorem}\label{theorem:serre-duality}
		Let $R$ be a local ring and $A$ a Gorenstein $R$-algebra with an isolated singularity. For $M,N$ in $\dsing(A)$ there is a natural isomorphism
		\[
		\Hom_{\dsing(A)}(M,N)^{\vee} \cong \Hom_{\dsing(A)}(N, \Sigma^{d-1}(\omega_{A/R}\lotimes_AM) )\,,
		\]
		where $d$ is the Krull dimension of $R$. Said otherwise,  the singularity category $\dsing(A)$ has Serre duality, with Serre functor $\Sigma^{d-1}(\omega_{A/R}\lotimes_A-)$. \qed
	\end{theorem}
	
	The preceding result is due to Buchweitz~\cite[Theorem~7.7.5]{Buchweitz:1986}, extending results of Auslander~\cite{Auslander:1978a} and Auslander and Reiten~\cite{Auslander--Reiten:1991}; see also  \cite[Corollary~9.3]{Iyengar--Krause:2021}.  One noteworthy consequence  is the existence of almost-split sequences in $\dsing(A)$; this is by the work of Reiten and Van den Bergh~\cite{Reiten--Van-den-Bergh:2002}.
\end{chunk}

\begin{chunk}\label{ch:relative-Serre}
	Several independent works 
	show that
	the 
	equivalence 
	\[
	\Sigma^d(\omega_{A/R} \lotimes_A -) \colon 
	{\dbcat A}
	\longiso  
	{\dbcat A}
	\]
	restricts to a Serre functor on a certain full subcategory of ${\dbcat A}$  whose morphism spaces are finite length $R$-modules,  see
	\cite[Theorem 7.2.14]{Ginzburg:2006}, \cite[Theorem 3.7]{Iyama--Reiten:2008}, \cite[Lemma 4.1]{Keller:2008} and  \cite[Lemma 6.4.1]{Van-den-Bergh:2004}.
\end{chunk}

The work in \cite{Iyengar--Krause:2021} started with the question
whether the Serre duality as in Theorem~\ref{theorem:serre-duality} is
a shadow of a more general duality statement, applicable to all
Gorenstein $R$-algebras and not just those with an isolated
singularity. There is indeed such a duality theorem, but to see it one
has to widen ones perspective to include also infinitely generated
analogues of maximal Cohen--Macaulay modules, namely the Gorenstein
projective modules, and develop the results presented above in that
generality. All this is explained in \cite{Iyengar--Krause:2021}.  This
endeavour was guided by Grothendieck duality theory in commutative
algebra and algebraic geometry, that subsumes and generalises
classical Serre duality in those contexts. Another inspiration was
the work of the second and third authors with Benson and
Pevtsova~\cite{Benson--Iyengar--Krause--Pevtsova:2019a,
	Benson--Iyengar--Krause--Pevtsova:2020a} on duality phenomena in the
stable module category of finite groups and group schemes.

\section{Gentle algebras} \label{se:gentle-algebras}
In this section we discuss gentle algebras. These arise frequently in
representation theory and were introduced by Assem and Skowro\'{n}ski
in their study of iterated tilted algebras of Dynkin type
$\mathbb{A}$ \cite{Assem--Skowronski:1987}.

Let $k$ be a field. A $k$-algebra is called
\emph{gentle} if it is Morita equivalent to an
algebra of the form $kQ/I$ where $Q$ is a finite quiver and $I$
is an ideal generated by paths of length two such that
the local shape at any vertex is one of the following:
\begin{align}
	\label{eq:vertices} 
	\quad 
	\begin{tikzcd}[row sep=scriptsize, column sep=scriptsize, cells={outer sep=1pt, inner sep=1pt}, ampersand replacement=\&] 
		\phantom{\circ} \arrow[rd, ""{name=A, inner sep=0, near end}] \& \& \phantom{\circ} \\
		\& \bullet \arrow[ld, ""{name=B, inner sep=0, near start}]
		\arrow[ru, ""{name=D, inner sep=0, near start}] \& \\
		\phantom{\circ} \& \& \phantom{\circ} \arrow[lu, ""{name=C, inner sep=0, near end}]
		\arrow[densely dotted, dash, bend right=60, from=A, to=B] \arrow[densely dotted, dash, bend right=60, from=C, to=D]
	\end{tikzcd}
	\begin{tikzcd}[row sep=scriptsize, column sep=scriptsize, cells={outer sep=1pt, inner sep=1pt}, ampersand replacement=\&] 
		\& \& \phantom{\circ} \\
		\phantom{\circ} \arrow[r, ""{name=A, inner sep=0, midway}] \& \bullet \arrow[ru, ""{name=B, inner sep=0, pos=0.45}]
		\arrow[rd, ""{name=D, inner sep=0, midway}] \& \\
		\& \& \phantom{\circ}
		\arrow[densely dotted, dash, bend left=60, from=A, to=B] \end{tikzcd}
	\begin{tikzcd}[row sep=scriptsize, column sep=scriptsize, cells={outer sep=1pt, inner sep=1pt}, ampersand replacement=\&]  
		\phantom{\circ} \arrow[rd, ""{name=A, inner sep=0, midway}] \& \&  \\
		\& \bullet \arrow[r, ""{name=B, inner sep=0, midway}]  \& \phantom{\circ} \\
		\phantom{\circ} \arrow[ru, ""{name=D, inner sep=0, midway}] \& 
		\arrow[densely dotted, dash, bend left=60, from=A, to=B] \& \end{tikzcd}
	\begin{tikzcd}[row sep=scriptsize, column sep=scriptsize, cells={outer sep=1pt, inner sep=1pt}, ampersand replacement=\&]
		\phantom{\circ} \& \\
		\phantom{\circ} \arrow[r, ""{name=E, inner sep=0.5, pos=0.45}] \& \bullet \arrow[r, ""{name=F, inner sep=0.5, pos=0.55}] \& \phantom{\circ} 
		\\
		\phantom{\circ} \& 
		\arrow[densely dotted, dash, bend left=60, from=E, to=F] \end{tikzcd}
	\begin{tikzcd}[row sep=scriptsize, column sep=scriptsize, cells={outer sep=1pt, inner sep=1pt}, ampersand replacement=\&] 
		\phantom{\circ} \arrow[rd, ""{name=A, inner sep=0, near start}] \&   \\
		\& \bullet \\
		\phantom{\circ} \arrow[ru, ""{name=D, inner sep=0, near start}] \& \end{tikzcd}
	\quad 
	\begin{tikzcd}[row sep=scriptsize, column sep=scriptsize, cells={outer sep=1pt, inner sep=1pt}, ampersand replacement=\&]  
		\&   \phantom{\circ} \\
		\bullet \arrow[ru, ""{name=A, inner sep=0, near start}] \arrow[rd, ""{name=D, inner sep=0, near start}] \& \\
		\& \phantom{\circ}   \end{tikzcd}
	\\
	\nonumber
	\begin{tikzcd}[row sep=0.2cm, column sep=scriptsize, cells={outer sep=1pt, inner sep=1pt}, baseline=0.05cm, ampersand replacement=\&]  
		\phantom{\circ} \arrow[r, ""{inner sep=0, near end}] \& \circ \arrow[r, ""{ inner sep=0, near start}] \& \phantom{\circ} 
	\end{tikzcd}
	\begin{tikzcd}[row sep=0.2cm, column sep=scriptsize, cells={outer sep=1pt, inner sep=1pt}, baseline=0.05cm, ampersand replacement=\&]  
		\phantom{\circ} \arrow[r, ""{inner sep=0, near end}] \& \circ 
	\end{tikzcd}
	\quad
	\begin{tikzcd}[row sep=0.2cm, column sep=scriptsize, cells={outer sep=1pt, inner sep=1pt}, baseline=0.05cm, ampersand replacement=\&]  
		{\circ} \arrow[r, ""{inner sep=0, near end}]  \& \phantom{\circ}
	\end{tikzcd}
\end{align}
Zero relations are indicated by dotted lines.
For later considerations
the vertices are divided into two types, depicted by full and hollow circles.

Rigorously formulated, the quiver $Q \colonequals (Q_0,Q_1,s,t)$
and the ideal $I$ are subject to the following conditions.
\begin{enumerate}[label=(Ge\arabic*),leftmargin=1.05cm]
	\item \label{eq:Ge1} For each  $i \in Q_0$, there are at most two arrows starting
	at $i$.
	\item \label{eq:Ge2} For each  $i\in Q_0$, there are at most two arrows ending
	at $i$.
	\item \label{eq:Ge3} For each $\alpha\in Q_1$, there is at most one arrow $\beta$ such
	that $s(\beta)=t(\alpha)$ and $\beta\alpha\in I$, and there is at most one arrow $\beta'$ such
	that $s(\beta')=t(\alpha)$ and $\beta'\alpha\not\in I$.
	\item \label{eq:Ge4} For each $\beta\in Q_1$, there is at most one arrow $\alpha$ such
	that $t(\alpha)=s(\beta)$ and $\beta\alpha\in I$, and there is at most one arrow $\alpha'$ with
	$t(\alpha')=s(\beta)$ and $\beta\alpha'\not\in I$.
\end{enumerate}
Note that in $Q$ we compose paths from right to left.  

\subsection*{Gentle algebras are Koszul}
\label{se:koszul}
Let $k$ be a field and $Q \colonequals (Q_0,Q_1,s,t)$ a finite quiver. Let $Q^{\mathrm{op}}$ be the \emph{opposite quiver} that
is obtained by reversing the arrows, so
\[
Q^{\mathrm{op}}_0\colonequals Q_0\qquad\textrm{and}\qquad Q^{\mathrm{op}}_1\colonequals \{\alpha^-\mid\alpha\in Q_1\}
\]
with $s(\alpha^-)\colonequals t(\alpha)$ and $t(\alpha^-)\colonequals s(\alpha)$.  Let $Q_2$ denote the set of paths of length $2$ and fix a partition $Q_2= P_+\sqcup P_-$.  It is
not difficult to check that the algebra $kQ/\langle P_- \rangle$ (not necessarily finite dimensional) is gentle if and only if $kQ^{\mathrm{op}}/\langle P^{\mathrm{op}}_- \rangle$ is gentle, where $P^{\mathrm{op}}_-\colonequals \{\alpha^-\mid\alpha\in P_+\}$.

Now  suppose that the algebra $A\colonequals kQ/\langle P_-\rangle$ is gentle. Then  $A$ is graded by path length and we have a decomposition
$A_0=\bigoplus_{i\in Q_0}S_i$ into simple $A$-modules. Each arrow $\alpha\colon i\to j$ corresponds to an extension
$\xi_\alpha\colon 0\to S_i\to E_\alpha\to S_j\to 0$. A theorem of Green and Zacharia \cite{Green--Zacharia:1994} shows that the algebra $A$ is Koszul, and the algebra $kQ^{\mathrm{op}}/\langle P^{\mathrm{op}}_-\rangle$ identifies with the Koszul dual $A^!\cong\Ext^*_{A}(A_0,A_0)$ via the assignment $\alpha^-\mapsto\xi_\alpha$. 

For example, consider the two-loop quiver with partitions
\[
\begin{tikzcd}[cells={outer sep=1.5pt, inner sep=1.5pt}, column sep =scriptsize]
	\bullet \arrow[looseness =6, out=30,in=-30,loop,"y",
	""{name=Y1, inner sep=0, near start},""{name=Y2, inner sep=0, near end}
	]
	\arrow[looseness =6, <-,out=150,in=-150,loop,swap,"x",""{name=X1, inner sep=0, near end},""{name=X2, inner sep=0, near start}
	]
	\arrow[densely dotted, dash, bend right=45, from=X1, to=X2] 
	\arrow[densely dotted, dash, bend right=45, from=Y1, to=Y2] 
\end{tikzcd}
\quad
P_+ = \{x^2, y^2\}\,, \qquad
\begin{tikzcd}[cells={outer sep=1.5pt, inner sep=1.5pt}, column sep =scriptsize]
	\bullet \arrow[looseness =6, out=30,in=-30,loop,"y",
	""{name=Y1, inner sep=0, very near start},""{name=Y2, inner sep=0, very near end}
	]
	\arrow[looseness =6, <-,out=150,in=-150,loop,swap,"x",""{name=X1, inner sep=0, very near end},""{name=X2, inner sep=0, very near start}
	]
	\arrow[densely dotted, dash, bend right=30, from=X1, to=Y2] 
	\arrow[densely dotted, dash, bend right=30, from=Y1, to=X2] 
\end{tikzcd}\quad
P_- = \{
xy,yx \}\,.
\]
Then the algebra $A\colonequals k\langle x,y\rangle/(x^2,y^2)$ is gentle with Koszul dual $A^!\cong 
k[x,y]/(xy)$.

\subsection*{Gluing gentle algebras}
For the remainder of this section, we fix an algebra 
$A \colonequals kQ/I$ such that $(Q,I)$ satisfies the above conditions
\ref{eq:Ge1}--\ref{eq:Ge4}.
Furthermore, we set $\fm \colonequals (c) \subset R \colonequals k[c]$.
The gentle algebra $A$ shares some properties 
with the path algebra $B \colonequals k Q'$ of a quiver defined as follows.

To obtain $Q'$, we split each full vertex from diagram \eqref{eq:vertices} into a pair of vertices:
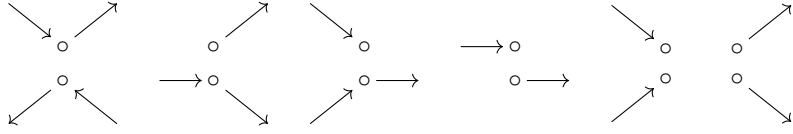
\begin{align*}
	\begin{tikzcd}[row sep=scriptsize, column sep=scriptsize, cells={outer sep=1pt, inner sep=1pt}, ampersand replacement = \&] 
		\phantom{\circ} \arrow[rd, ""{name=A, inner sep=0, near end}] \& \& \phantom{\circ} \\
		\& \circ 
		\arrow[ru, ""{name=B, inner sep=0, near start}] \& \\[-0.25cm]
		\& \circ 
		\arrow[<-, rd, ""{name=D, inner sep=0, near start}] \& \\
		\phantom{\circ}\arrow[<-,ru, ""{name=C, inner sep=0, near end}] \& \& \phantom{\circ}
	\end{tikzcd}
	\begin{tikzcd}[row sep=scriptsize, column sep=scriptsize, cells={outer sep=1pt, inner sep=1pt}, ampersand replacement = \&] 
		\& \& \phantom{\circ} \\
		\phantom{\circ}  \& \circ \arrow[ru, ""{name=B, inner sep=0, near start}]  
		\& \\[-0.25cm]
		\phantom{\circ} \arrow[r, ""{name=A, inner sep=0, midway}] \&  \circ 
		\arrow[rd, ""{name=D, inner sep=0, midway}] \& \\
		\& \& \phantom{\circ}
	\end{tikzcd}
	\begin{tikzcd}[row sep=scriptsize, column sep=scriptsize, cells={outer sep=1pt, inner sep=1pt}, ampersand replacement = \&] 
		\phantom{\circ} \arrow[rd, ""{name=A, inner sep=0, midway}] \& \&  \\
		\& \circ   
		\& \phantom{\circ} \\[-0.25cm]
		\& \circ \arrow[r, ""{name=B, inner sep=0, midway}] \&  \phantom{\circ} \\
		\phantom{\circ} \arrow[ru, ""{name=B, inner sep=0, midway}]   \& \phantom{\circ}  \& \phantom{\circ} 
	\end{tikzcd}
	\begin{tikzcd}[row sep=scriptsize, column sep=scriptsize, cells={outer sep=1pt, inner sep=1pt}, ampersand replacement = \&] 
		\phantom{\circ} \arrow[r, ""{name=E, inner sep=0, midway}] \& \circ 
		\& 
		\\[-0.25cm]
		\& \circ \arrow[r, ""{name=F, inner sep=0, midway}] \& \phantom{\circ} 
	\end{tikzcd}
	\begin{tikzcd}[row sep=scriptsize, column sep=scriptsize, cells={outer sep=1pt, inner sep=1pt}, ampersand replacement = \&] 
		\phantom{\circ} \arrow[rd, ""{name=A, inner sep=0, near start}] \&   \\
		\& \circ  
		\\[-0.3cm]
		\& \circ \\
		\phantom{\circ} \arrow[ru, ""{name=D, inner sep=0, near start}] \&  \end{tikzcd}
	\qquad
	\begin{tikzcd}[row sep=scriptsize, column sep=scriptsize, cells={outer sep=1pt, inner sep=1pt}, ampersand replacement = \&] 
		\& \phantom{\circ} \\
		\circ 
		\arrow[ru, ""{name=A, inner sep=0, near start}] \&   \phantom{\circ} \\[-0.3cm]
		\circ \arrow[rd, ""{name=D, inner sep=0, near start}] \& \\
		\& \phantom{\circ}  \end{tikzcd}
\end{align*}
Formally, the quiver
$Q' \colonequals  (Q'_0, Q'_1, s',t')$ is defined as follows.
Let $Q^*_0$ denote the subset of all vertices $i \in Q_0$ such that 
no arrow of $Q$ starts or ends in $i$.
Set
$Q'_0 \colonequals Q^*_0 \cup \{ s_{\alpha}, t_{\alpha} \mid \alpha \in Q_1\}/{\sim} $,
where $\sim$ is the smallest equivalence relation 
such that 
$s_{\beta} \sim t_{\alpha}$ 
for any 
$\alpha,\beta \in Q_1$ with $s(\beta) = t(\alpha)$ and $\beta \alpha \notin I$. Set $Q'_1 \colonequals Q_1$, $s'(\alpha) \colonequals  [s_\alpha]$, and $t'(\alpha)\colonequals  [t_{\alpha}]$ for any $\alpha \in Q_1'$.

The ring $B$ is a central tool in the work by Burban and Drozd \cite{Burban--Drozd:2004, Burban--Drozd:2017, Burban--Drozd:2019} to study the representation theory of $A$.
The statements of the next lemma are known in similar settings; see \cite[Lemma 2.12]{Burban--Drozd:2017} and \cite[Theorem 3.11]{Burban--Drozd:2019}.
\begin{lemma}\label{le:ovr}
	With the notation above, the following statements hold.
	\begin{enumerate}
		\item \label{le:ovr1}
		The quiver $Q'$ is a finite union of quivers of equioriented type $\mathbb{A}$ or $\widetilde{\mathbb{A}}$.
		In different terms, 
		there are 
		numbers $0 \leq m \leq n$, $\ell_1, \ldots, \ell_n \geq 1$ such that there
		is an isomorphism of rings
		$B \cong \prod_{v=1}^m T_{\ell_v}(k) \times \prod_{v=m+1}^{n} T_{\ell_v}(R)$, where
		\begin{align*}
			T_{\ell_v}(k)& \colonequals  \{ (a_{ij}) \in \Mat_{\ell_v \times \ell_v}(k) \mid a_{ij} = 0 \text{ for any }1 \leq i < j \leq \ell_v \},\\
			T_{\ell_v}(R) &\colonequals  \{ (a_{ij}) \in \Mat_{\ell_v \times \ell_v}(R) \mid a_{ij} \in \fm \text{ for any }1 \leq i < j \leq \ell_v \}\,.
		\end{align*}
		\item \label{le:ovr2}
		Let $J_A$ denote the two-sided ideal of $A$ generated by the arrows of its quiver.
		Similarly, let $J_B$ denote the  arrow ideal of $B$.
		There is a canonical monomorphism $\iota\colon A \hookrightarrow B$  of $k$-algebras such that $\iota(J_A) = J_B$. 
	\end{enumerate}
\end{lemma}
\begin{proof}
	\eqref{le:ovr1}
	Properties 
\ref{eq:Ge3} and \ref{eq:Ge4}
	 of $(Q,I)$ ensure that 
	there is at most one arrow starting and at most one arrow ending at each vertex in $Q'$. 
	
	\eqref{le:ovr2}
	The morphism $\iota$ in question is given by the unique $k$-algebra morphism with  
	$\alpha \mapsto {\alpha}$  for $\alpha \in Q_1$,
	$e_i \mapsto e_i$ for $i \in Q_0^*$, and	
	$e_i \mapsto \sum_{j \in Q'_0(i)} e_{j}$ 
	for $i \in Q_0 \backslash Q_0^*$,
	where
	\begin{align*}
		Q'_0(i) \colonequals \{[s_{\beta}] \mid \beta \in Q_1\colon s(\beta) = i \} \cup \{[t_{\alpha}] \mid \alpha \in Q_1\colon t(\alpha) = i \} 
	\end{align*}
	is a subset of $Q'_0$ with only one or two elements.
\end{proof}
The above construction can be reversed: given a union 
$Q'$ of equioriented quivers of type $\mathbb{A}$ or $\widetilde{\mathbb{A}}$, one may 
identify certain pairs of vertices adopting the relations of the arrows in $Q'$ to glue a gentle quiver $(Q,I)$.
In fact, any gentle quiver can be obtained in this way.

\subsection*{Completions of gentle algebras}
The dimension of the gentle algebra $A$ is
related to the following notion.
For any arrow  $\alpha \in Q_1$,
there is a unique maximal repetition-free path $p_{\alpha} = \alpha_n \ldots \alpha_1 \notin I$ beginning with $\alpha_1 \colonequals \alpha$ by 
\ref{eq:Ge3}.
In other words, all $n$ arrows of $p_{\alpha}$ are different, $p_{\alpha} \notin I$ and for any arrow $\beta \in Q_1$ with $s(\beta) = t(\alpha_n)$ it follows that $\beta = \alpha_1$ 
or $\beta \alpha_n \in I$.
We call $p_{\alpha}$ an \emph{admissible cycle} if it is a cyclic path with $\alpha_1 \alpha_n \notin I$.

Next, we consider certain completions of an infinite dimensional gentle algebra.
\begin{lemma}\label{le:global-gentle}
	Assume that the quiver $(Q,I)$ has an admissible cycle.
	The following statements hold.
	\begin{enumerate}
		\item \label{le:gg1} There is a canonical monomorphism $\iota\colon A \hookrightarrow B$ 
		of finite $R$-algebras.
		\item \label{le:gg2} For any $\fp \in \Spec R\backslash\{{\fm}\}$ 
		there is an isomorphism $\widehat{A}_{\fp} \cong 
		\prod_{v=1}^n
		\Mat_{\ell_v \times \ell_v}(\widehat{R}_{\fp})
		$
		of finite $\widehat{R}_{\fp}$-algebras, where $n, \ell_1, \ldots, \ell_n \geq 1$.
		\item \label{le:gg3} The ring $\widehat{A}_{\fm}$ is a finite $\widehat{R}_{\fm}$-algebra
		which is isomorphic to the completion of the path algebra $A$ with respect to its arrow ideal $J_A$ and the $Q(\widehat{R}_{\fm})$-algebra $\widehat{A}_{\fm} \otimes_{\widehat{R}_{\fm}}Q(\widehat{R}_{\fm})$ is semisimple.
	\end{enumerate}
\end{lemma}
\begin{proof}
	\eqref{le:gg1}
	Identifying $c$ in $R = k[c]$ with the sum of all admissible cycles in $(Q,I)$ yields an $R$-algebra map $\phi\colon R \to A$.
	The composition $\iota \phi$ with the embedding $\iota$ from Lemma~\ref{le:ovr}~\eqref{le:ovr2} 
	maps $c$ to the sum of all repetition-free cyclic paths in $Q'$ and  
	endows the overring $B$ with the structure of a finite $R$-algebra.
	Viewing $\iota\colon A \hookrightarrow B$ as monomorphism of $R$-modules, it follows that $A$ is a finite $R$-algebra as well.
	
	\eqref{le:gg2}
	Let $\fp \in {\Spec R}\backslash\{\fm\}$.
	It holds that $c B \subseteq \iota(A)$ and $c \, T_{\ell_v}(k) = 0$ for any ring factor $T_{\ell_v}(k)$ of $B$ as well as $c \notin \fp$. 
	These observations imply that $\iota_{\fp}(A_{\fp}) = B_{\fp} \cong \prod_{v=1}^n (T_{\ell_v}(R))_{\fp}
	\cong \prod_{v=1}^n \Mat_{\ell_v \times \ell_v}(R_{\fp})$
	with the notation of Lemma~\ref{le:ovr}~\eqref{le:ovr1}. Taking completions
	yields statement \eqref{le:gg2}.
	
	\eqref{le:gg3} The map $\iota$ gives rise to a monomorphism of algebras $\widehat{\iota}_{\fm}\colon \widehat{A}_{\fm} \hookrightarrow \widehat{B}_{\fm}$ 
	over the ring $\widehat{R}_{\fm} \cong k \llbracket c \rrbracket$ of formal power series.
	Since $\widehat{\iota}_{\fm}$ has finite dimensional cokernel, 
	$\widehat{A}_{\fm} \otimes_{k\llbracket c \rrbracket} k (\!(c)\!)  \cong \widehat{B}_{\fm} \otimes_{k \llbracket c \rrbracket} k (\!(c)\!)$ is a semisimple $k (\!(c)\!)$-algebra. 
	Since $J_A^n \subseteq A\fm  \subseteq J_A$ for certain $n \geq 1$, it follows that $\widehat{A}_{\fm}$ is $J_A$-adically complete.
\end{proof}
The last lemma shows that an infinite dimensional gentle algebra $A$ has only one interesting completion at a prime ideal of its central subring $R$.
A simpler analogue of the last lemma holds for finite dimensional gentle algebras, in which the ring $R$ has to be replaced by the field $k$ and $\widehat{A}_{\fm}$ can be identified with the algebra $A$ itself.

\subsection*{Finite projective gentle algebras}

At this point, we can clarify when the gentle algebra $A$ 
satisfies property~\ref{de:gor-algebra}~\eqref{Gorenstein-fp}
of a Gorenstein $R$-algebra over the ring $R = k[c]$.
\begin{proposition} \label{pr:fp} With the notation above, the following are equivalent.
	\begin{enumerate}
		\item \label{pr:fp1} $A$ admits the structure of a finite projective $R$-algebra;
		\item \label{pr:fp2} $B$ admits the structure of a finite projective $R$-algebra;
		\item \label{pr:fp3} Each arrow of $Q'$ lies on a cyclic path;
		\item \label{pr:fp4} Each arrow of $Q$ lies on an admissible cycle.
	\end{enumerate}
\end{proposition}
\begin{proof} 
	\eqref{pr:fp1}$\Rightarrow$\eqref{pr:fp2} 
	Because of~\eqref{pr:fp1}, $Q_0^*$ is empty and $A$ has no ring factor isomorphic to $k$.
	Together with Lemma~\ref{le:ovr}~\eqref{le:ovr2}
	assumption~\eqref{pr:fp1} implies that $B$ can be viewed as a finite $R$-module such that its arrow ideal $J \colonequals J_B$ is torsion-free as an $R$-module.
	Since $J \, {\tor}_R (B) + {\tor}_R(B)  \, J \subseteq {\tor}_R(J) = 0$ and $B$ has no simple ring factor, 
	$B$ must be also torsion-free, and thus projective over $R$.
	Note that a triangular matrix ring of the form $T_{\ell_v}(k)$ cannot be endowed with the structure of a projective $R$-module.
	Therefore, each indecomposable ring factor of $B$ is isomorphic to a matrix ring
	of the form $T_{\ell_v}(R)$, which is equivalent to property~\eqref{pr:fp2}.
	
	\eqref{pr:fp2}$\Rightarrow$\eqref{pr:fp3}
	follows from the just mentioned characterisation of \eqref{pr:fp2}.
	
	\eqref{pr:fp3}$\Leftrightarrow$\eqref{pr:fp4} holds by
	the definition of the quiver $Q'$. 
	
	\eqref{pr:fp3}$\Rightarrow$\eqref{pr:fp1}
	Because of \eqref{pr:fp3}, every ring factor of $B$ is given by a matrix ring of the form $T_{\ell_v}(R)$.
	We may repeat the proof of 
	Lemma~\ref{le:global-gentle}~\eqref{le:gg1}
	adding that $\iota \phi$ endows the ring $B$ with the structure of a finite projective $R$-algebra.
	Thus its $R$-subalgebra $A$ is finite and projective as $R$-module as well.
\end{proof}

\subsection*{Gentle algebras are Iwanaga--Gorenstein} This was proved by Gei{\ss} and Reiten~\cite{Geiss--Reiten:2005} when the algebra is also finite dimensional; the general case is handled in \cite[Proposition 6.2.9]{Krause:2021}. In fact, the injective dimension can be computed explicitly.

\begin{chunk}\label{ch:differential}
	Koszul duality suggests a concept analogous to a maximal repetition-free path without relations.
	Namely, for any arrow $\alpha \in Q_1$ \ref{eq:Ge3} implies that there is a unique path
	$d_{\alpha} =  \alpha_n \ldots \alpha_2\, \alpha_1$ of pairwise distinct arrows
	beginning with  $\alpha_1 \colonequals \alpha$
	such that
	$s(\alpha_{i+1}) = t(\alpha_i)$ and
	$\alpha_{i+1} \alpha_i \in I$ for any $1 \leq i < n$, and for any arrow $\beta \in Q_1$ with $s(\beta) = t(\alpha_n)$ it follows that $\beta = \alpha_1$ or $\beta \alpha_n \notin I$. 
	We call $d_{\alpha}$
	a \emph{differential cycle} if it is a cyclic path with $\alpha_1 \alpha_n \in I$,  
	and a \emph{differential walk}
	otherwise.
\end{chunk}

We define an integer $w(A)$ for the gentle algebra $A = kQ/I$ as follows.
Set $w(A)$ to be equal to the maximum of all lengths of all differential walks in $(Q,I)$ if there is at least one differential walk, 
$w(A) \colonequals 0$ if each indecomposable ring factor of $A$ is isomorphic to $k$
or to the path algebra of a finite quiver of equioriented type $\widetilde{\mathbb{A}}$ modulo the square of its arrow ideal,
and $w(A) \colonequals 1$ in the remaining cases.
\begin{proposition}\label{pr:gentle-gorenstein}
	The gentle algebra $A$ is Iwanaga--Gorenstein. 
	Its injective dimension is given by $w(A)$.\qed
\end{proposition}

\begin{remark}
	The only obstacle of the gentle algebra $A$ being a Gorenstein  $R$-algebra is that $A$ may not be projective as an $R$-module.
	Crawley-Boevey pointed out that this happens for the
	path algebra of the quiver $(Q^*,I^*)$ in Example~\ref{example:hereditary}.
	This again suggests that one should relax condition  \eqref{eq:fp} in \ref{de:gor-algebra} to allow $A$ to have finite projective  dimension over $R$; see the discussion on the work of Buchweitz at the end of Section~\ref{se:Gorenstein-algebras}.
\end{remark}
In the remainder of this paper we will focus on arrow ideal completions of infinite dimensional gentle Gorenstein algebras, which will be called \emph{gentle Gorenstein algebras} for brevity.

\section{Serre duality for gentle Gorenstein algebras}
\label{se:serre-gentle}

From now on, let $(Q,I)$ be a gentle quiver such that each arrow lies on an admissible cycle.
Equivalently, $Q$ is a finite quiver, 
the ideal $I$ is generated by paths of length two and each vertex is 
either
\emph{2-regular} with gentle relations
or a \emph{transition vertex}:
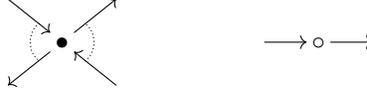
\[\
\begin{tikzcd}[row sep=scriptsize, column sep=scriptsize, cells={outer sep=1pt, inner sep=1pt}, ampersand replacement=\&] 
	\phantom{\circ} \arrow[rd, ""{name=A, inner sep=0, near end}] \& \& \phantom{\circ} \\
	\& \bullet \arrow[ld, ""{name=B, inner sep=0, near start}]
	\arrow[ru, ""{name=D, inner sep=0, near start}] \& \\
	\phantom{\circ} \& \& \phantom{\circ} \arrow[lu, ""{name=C, inner sep=0, near end}]
	\arrow[densely dotted, dash, bend right=60, from=A, to=B] \arrow[densely dotted, dash, bend right=60, from=C, to=D]
\end{tikzcd}
\qquad \qquad
\begin{tikzcd}[row sep=scriptsize, column sep=scriptsize, cells={outer sep=1pt, inner sep=1pt}] 
	\phantom{\circ} 
	& & \phantom{\circ} \\
	\phantom{\circ} \arrow[r, ""{inner sep=0, near end}] & \circ \arrow[r, ""{ inner sep=0, near start}] & \phantom{\circ}  \\
	\phantom{\circ}
	& & \phantom{\circ} 
\end{tikzcd}
\]
In contrast to the previous section,
we denote  by $R$ the ring $k \llbracket c \rrbracket$ of formal power series
and by $A$ the completion of the
infinite dimensional path algebra $k Q/I$ with respect to its arrow ideal.
According to Lemma~\ref{le:global-gentle} \eqref{le:gg3} and
Proposition~\ref{pr:fp}, the ring $A$ is a gentle Gorenstein $R$-algebra.

The purpose of this section is to give an explicit description of the dualising bimodule $\omega_{A/R} \colonequals \Hom_R(A,R)$. 
Its computation
is similar in spirit to work of Ladkani \cite{Ladkani:2012}
and has been carried out in \cite{Gnedin:2019} for the case that the quiver $Q$ has $2$-regular vertices only.
The description of the bimodule $\omega_{A/R}$ requires some notation, which will not be relevant in the next section.
The main 
conclusion of this description is summarised in Corollary~\ref{corollary:Nakayama-functor}.

\begin{chunk}
	For any arrow $\alpha \in Q_1$ there is a unique arrow $\sigma(\alpha) \in Q_1$ such that $\sigma(a) a \notin I$.
	Moreover, $\alpha$ is the beginning of a unique admissible cycle $c_\alpha$ of a certain length $\ell_{\alpha} \geq 1$. 
	We may view the ring $A$ as a Gorenstein $R$-algebra via the structure map
	\begin{align}
		\label{eq:structure-map}
		R = k \llbracket c \rrbracket \longrightarrow A, \quad c\longmapsto \sum_{\alpha \in Q_1} c_\alpha
	\end{align}
	For any integer $1 \leq n < \ell_{\alpha}$ we denote by $\alpha_n$ 
	the unique path of length $n$ which begins with the arrow $\alpha$ and is not contained in the ideal $I$, and by 
	$c_\alpha \partial^n$ the unique path such that
	$(c_\alpha\partial^n)  \alpha_n = c_\alpha$. The path $c_\alpha \partial^n$ can also be obtained by deleting the first $n$ arrows in the path $c_\alpha$. When using these notation, we assume implicitly that $\alpha \neq c_{\alpha}$.
	Let $Q^r_0$ and $Q^t_0$ denote the set of $2$-regular vertices and
	the set of transition vertices, respectively.
	Let $Q^t_1$ be the set of all arrows starting in transition vertices.
	
	To fix an $R$-linear basis of $A$, we choose a map $\varepsilon\colon Q_1 \to \{-1,1\}$, $\alpha \mapsto \varepsilon_{\alpha} \colonequals \varepsilon(\alpha)$, such that $\varepsilon_{\alpha} \neq \varepsilon_{\beta}$ for any distinct arrows $\alpha, \beta$ with $s(\alpha) = s(\beta)$, and $\varepsilon_{\gamma} =  -1$ for any arrow $\gamma \in Q^t_1$.
	Finally, let $Q_1^+$ denote the set of all arrows $\alpha$ with $\varepsilon_{\alpha} = 1$.
\end{chunk}
\begin{example}\label{example:prototype}
	The following quiver is glued from cycles of lengths two and nine:
	
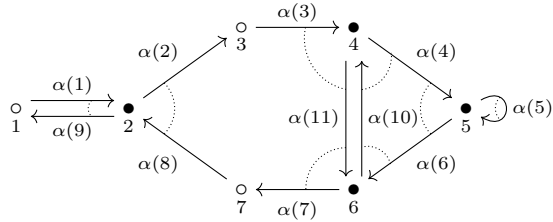
\[
	\begin{tikzcd}[row sep=0.6cm, column sep=1.25cm, cells={outer sep=2pt, inner sep=1pt}] 
		&& 
		\underset{3}{\circ} \ar[r, "\alpha(3)", ""{name=A3, inner sep=0, pos=0.6}] & \underset{4}{\bullet} \ar[xshift=-3pt,	"\alpha(11)"{swap}, 
		""{name=B1s, inner sep=0, near start}, ""{name=B1t, inner sep=0, near end}]{dd} \ar[rd, "\alpha(4)", ""{name=A4s, inner sep=0, near start}, ""{name=A4t, inner sep=0, near end}] & 
		\\
		\underset{1}{\circ} \ar[r, yshift=3pt, "\alpha(1)",""{name=A1, inner sep=0, near end}] & 
		\underset{2}{\bullet} \ar[ru, "\alpha(2)",""{name=A2, inner sep=0, near start}] 
		\ar[l, "\alpha(9)", yshift=-3pt, ""{name=A9, inner sep=0, near start}] 
		&&& 
		\underset{5}{\bullet} 
		\ar[ld, "\alpha(6)",""{name=A6s, inner sep=0, near start}, ""{name=A6t, inner sep=0, near end}] 
		\ar[looseness=6, out=30, in=-30, "\alpha(5)", ""{name=A5s, inner sep=0, near start}, ""{name=A5t, inner sep=0, near end}]
		\\
		&& 
		\underset{7}{\circ} \ar[lu, "\alpha(8)",""{name=A8, inner sep=0, near end}] & 
		\underset{6}{\bullet} \ar[l, "\alpha(7)"{yshift=-2pt},""{name=A7, inner sep=0, pos=0.4}] \ar[xshift=3pt, "\alpha(10)"{swap},""{name=B2s, inner sep=0, near start}, ""{name=B2t, inner sep=0, near end}]{uu} & 
		\arrow[densely dotted, dash, bend right=45, from=A1, to=A9]
		\arrow[densely dotted, dash, bend left=45, from=A2, to=A8]
		\arrow[densely dotted, dash, bend right=45, from=A3, to=B1s]
		\arrow[densely dotted, dash, bend right=45, from=B1t, to=A7]
		\arrow[densely dotted, dash, bend right=45, from=B2t, to=A4s]
		\arrow[densely dotted, dash, bend right=45, from=A4t, to=A6s]
		\arrow[densely dotted, dash, bend right=45, from=A6t, to=B2s]
		\arrow[densely dotted, dash, bend right=45, from=A5s, to=A5t]
	\end{tikzcd}
	\]
	For this quiver, we may choose $\varepsilon_{\alpha(i)} \colonequals  (-1)^i$ for $i \neq 8$ and $\varepsilon_{\alpha(8)} \colonequals -1$.
\end{example}

It is not hard to verify the following.
\begin{lemma}
	As an $R$-module $A$ has an $R$-linear basis 
	\[
	\qquad \quad
	\mathcal{B} \colonequals  \{ e_i \mid i \in Q_0 \} \cup \{ \alpha_n \mid \alpha \in Q_1, \ 1 \leq n < \ell_{\alpha} \} \cup \{ c_{\beta} \mid \beta \in Q_1^+ \}\,. 
	\qquad \quad \qed
	\]
\end{lemma}
For any path $p \in \mathcal{B}$ let $p^{*}\colon A \to R$ denote the unique $R$-linear map such that for any $q \in \mathcal{B}$ it holds that $p^{*}(q) = \delta_{pq}$, which denotes the Kronecker delta.

Each vertex $i \in Q_0$ induces an indecomposable projective $A$-module $P_i \colonequals A e_i$.  The intersection of its maximal left submodules is denoted $\rad {P_i}$. The left $A$-module structure of the dualising bimodule is almost regular in the following sense.

\begin{proposition}
	There is an isomorphism of left $A$-modules
	\[
	\vartheta\colon A^{\circ} \colonequals  \big(\bigoplus_{i \in Q^r_0} P_i\big) \oplus \big(\bigoplus_{j \in Q^t_0} 
	{\rad {P_j}}\big) \longiso \omega_{A/R} = \Hom_R(A,R)\,.
	\]
	More precisely, $\vartheta$ is the unique left $A$-linear map such that
	\[
	e_i \longmapsto c_{\beta(i)}^{*} \text{ for any }i \in Q^r_0\,,
	\quad  \text{and} \quad
	\gamma(j) \longmapsto 
	- (c_{\gamma(j)}\partial)^{*} \text{ for any }j\in Q^t_0
	\]
	where
	$\beta(i)$ denotes the unique arrow in $Q_1^+$ starting in $i$
	and ${\gamma(j)}$  the unique arrow starting in $j$.
\end{proposition}

\begin{proof}
	First, we compute $\vartheta$ on an $R$-linear basis of $P_i$ for $i \in Q^r_0$.
	Set $c_i \colonequals  c_{\beta(i)}$.  For any $p \in \mathcal{B}$  it holds that 
	$\vartheta(p e_i) = p \cdot \vartheta(e_i) = p \cdot c^{*}_{i} = c^{*}_{i}(\mathstrut_{-} p) = \sum_{q \in \mathcal{B}} c^{*}_{i}(qp) \, q^{*}.$
	It can be shown that $c^{*}_{i}(qp) \neq 0$  if and only if $qp = c_i^2$ or $c_\alpha$ for an arrow $\alpha$ starting at $i$. 
	In this case, it holds that
	$c^{*}_{i}(qp) = c$ or $\varepsilon_{\alpha}$, respectively.
	It follows that $\vartheta(e_i) = e_i \cdot \vartheta(e_i)$, 
	$\vartheta(c_i) = e_i^{*} + c \, c_{i}^{*}$ for $i \in Q^r_0$, and 
	$\vartheta(\alpha_n) = \varepsilon_{\alpha} (c_{\alpha}{\partial^n})^*$ for any $\alpha \in Q_1\backslash Q^t_1$ and $1 \leq n < \ell_{\alpha}$.
	
	Next, we describe $\vartheta$ on an $R$-linear basis of $\rad{P_j}$
	for $j \in Q^t_0$. Set $\gamma \colonequals \gamma(j)$.  
	For any $p, q \in \mathcal{B}$ it holds that
	$(c_{\gamma} \partial)^*(qp) \neq 0$ if and only if $qp = c_\gamma \partial$, which may occur only if 
	$p = e_{t(\gamma)}$ or $\sigma(\gamma)_{n}$ with $1 \leq n < \ell_{\gamma}$.
	By similar arguments as in the previous case, it follows that
	$e_{t(\gamma)} \cdot \vartheta(\gamma) =  \vartheta(\gamma)$, $\vartheta(\gamma_n) = \sigma(\gamma)_{n-1} \cdot \vartheta(\gamma) = - (c_{\gamma} \partial^n)^*$ for $1 < n < \ell_{\gamma}$
	and $\vartheta(c_{\gamma}) = c_{\gamma} \partial\cdot \vartheta(\gamma) = - e^*_j$.
	This shows also that $\vartheta$ is well-defined.
	
	Let $\psi\colon \omega_{A/R}\to A^{\circ}$ be the unique $R$-linear map such that
	$e_i^{*} \mapsto c_{\beta(i)} - c e_i$ for $i \in Q^r_0$, $e_j^{*} \mapsto -c_{\gamma(j)}$ for $j \in Q^t_0$,
	$c_{\beta}^{*} \mapsto e_{s(\beta)}$ for $\beta \in Q_1^+$ and $\alpha_n^{*} \mapsto \varepsilon_{\sigma^{n}(\alpha)} c_{\alpha}\partial^n$ for $\alpha \in Q_1$ and $1 \leq n < \ell_{\alpha}$.   
	It can be checked that $\psi$ is the inverse of the map $\vartheta$.
\end{proof}

\begin{chunk}
	To describe the right $A$-module structure of $\omega_{A/R}$ we need a few more notions.
	Let $\nu \colonequals  
	\widehat{\nu}_{\varepsilon}\colon A \to A$ be 
	the completion of 
	the unique $k$-algebra morphism $\nu_\varepsilon\colon k Q/I \to k Q/I$
	such that $ e_i \mapsto e_i$ for any $i \in Q_0$ and $\alpha \mapsto \varepsilon_{\sigma(\alpha)}\varepsilon_{\alpha} \alpha$ for any $\alpha \in Q_1$.
	In particular, each path $p$ in $(Q,I)$ coincides with $\nu(p)$ up to sign. 
	
	Let $A_{\nu}$ be the $A$-bimodule with regular left $A$-module  structure  and right $A$-module structure twisted by $\nu$, that is, we set $x \cdot_{\nu} \, a\colonequals  x\,\nu(a)$ for any $x \in A_{\nu}$ and $a \in A$.
	Let $A_{\nu}^{\circ}$ be the subbimodule of $A_{\nu}$ generated by 
	the idempotents $e_i$ with $i \in Q^r_0$
	and the arrows  $\gamma \in Q^t_1$. 
	This twisted bimodule is the dualising bimodule.
\end{chunk}

\begin{theorem}\label{theorem:canonical-bimodule}
	The left $A$-linear isomorphism $\vartheta\colon A^{\circ}_\nu \longiso \omega_{A/R}$ is $A$-bilinear.
\end{theorem}
\begin{proof}
	For any vertex $i \in Q_0^r$ we set $c_i \colonequals c_{\beta(i)}$ where $\beta(i)$ is the arrow in $Q_1^+$ starting at $i$.
	Fix $i \in Q^r_0$ and $p \in\mathcal{B}$. Set $h \colonequals s(p)$.
	\begin{itemize}[label=--, leftmargin=*]
		\item
		If $t(p) \neq i$, it holds that
		$\vartheta(e_i) \, p = 0 = \vartheta(e_i \cdot_{\nu} p)$. 
		\item
		Assume that $t(p)= i$.
		The type of the vertex $h$ leads to the following cases.
		\begin{itemize}[label=--, leftmargin=*]
			\item
			Assume that $h \in Q_0^r$.
			For any $q \in \mathcal{B}$ it can be verified that
			$c^*_{h}(q\,\nu(p)) = 
			c^*_i(pq)$.
			Thus, 
			$
			\vartheta(e_i \cdot_{\nu} p) = 
			\vartheta(\nu(p) \, e_h) =c_{h}^*(\mathstrut_- \nu(p))
			= \sum_{q \in \mathcal{B}} c_h^*(q\,\nu(p))\, q^*$
			$= c_{i}^*(p \mathstrut_{-}) =
			\vartheta(e_i)\, p $. 
			\item 
			Otherwise, $h \notin Q_0^r$.
			Then $p = \alpha_n$ for an arrow $\alpha \in Q_1^t$ and $1\leq n < \ell_{\alpha}$.
The equality $c_i^*(p \, (c_\alpha \partial^n) ) = \varepsilon_{\sigma^n(\alpha)}$ implies 
$\vartheta(e_i \cdot_\nu p) 
			= -\varepsilon_{\sigma^{n}(\alpha)}\,\vartheta(\alpha_n)
			=
			\varepsilon_{\sigma^{n}(\alpha)}\, (c_{\alpha} \partial^n)^*
			= \sum_{q \in \mathcal{B}} c_i^*(p\,q)\, q^*
			= c_{i}^*(p \mathstrut_{-})
			= \vartheta(e_i) p$.
		\end{itemize}
	\end{itemize}
	Let $\gamma \in Q_1^t$. Set $j \colonequals s(\gamma)$ and $\varepsilon \colonequals \varepsilon_{\sigma(\gamma)}$.
	It holds that
	$\vartheta(\gamma\cdot_{\nu} p) =0= \vartheta(\gamma) \cdot p$
	for any $p \in \mathcal{B}$ with $t(p) \neq j$ and 
	$\vartheta(\gamma \cdot_{\nu} e_{j}) =
	\vartheta(\gamma) 
	= \vartheta(\gamma) \,e_j$. Furthermore, it can be computed that
	$\vartheta(\gamma\cdot_{\nu} c_{\gamma}\partial^n)  = -\sigma(\gamma)_{n-1} = \vartheta(\gamma)\cdot c_{\gamma}\partial^n$
	for any  $1 < n < l_{\alpha}$
	and $\vartheta(\gamma\cdot_{\nu} c_{\gamma}\partial)  = 
	- e_{t(\gamma)}^* - \delta_{\varepsilon, +} \, c \, c^*_{\sigma(\gamma)} 
	= \vartheta(\gamma)\cdot c_{\gamma} \partial$.
	As $\vartheta$ is $R$-linear the claim follows.
\end{proof}

This theorem yields an explicit description of the relative Serre
functor of the category $\dbcat A$ mentioned in
\ref{ch:relative-Serre}.  We will not need the precise formulation of
Theorem~\ref{theorem:canonical-bimodule} in the following, but only an
approximate version.
\begin{corollary}\label{corollary:Nakayama-functor}
	For any complex $P$ of projective $A$-modules the complex $\omega_{A/R} \otimes_A P$ is given by replacing each indecomposable module $P_j$ with $j \in Q^t_0$ in $P$ with its radical
	and applying a certain sign change to each path in each differential. \qed
\end{corollary}
In the applications below, the sign change will not affect the isomorphism class. 

\section{Singularity categories of gentle Gorenstein algebras}
\label{se:singularity-categories}

As in the previous section, we study the arrow ideal completion $A$ of the path algebra of a gentle quiver such that each of its arrows lies on an admissible cycle, and view $A$
as an $R$-algebra over the ring $R= k \llbracket c \rrbracket$
via the structure map \eqref{eq:structure-map}.

In this section, we give a description of the singularity category $\dsing(A)$ in Theorem~\ref{th:dsing-gentle}, which is an analogue of a result by Kalck for finite dimensional gentle algebras \cite{Kalck:2015}.
The resulting description of $\dsing(A)$ also follows from
a classification of indecomposable objects
in the equivalent homotopy category of acyclic projective complexes by work of Bennett-Tennenhaus \cite{Bennett-Tennenhaus:2017}
or of those in the bigger category $\dbcat A$ using techniques by Burban and Drozd developed in \cite{Burban--Drozd:2004} and \cite{Burban--Drozd:2017}.

Our approach is based on the representation theory of lattices over orders. A key point is that
Corollary~\ref{corollary:Nakayama-functor}
allows to compute the Auslander-Reiten translations of non-projective $A$-lattices.
After determining the whole Auslander-Reiten quiver of $A$-lattices we will translate classification results along the chain of categories
\[
\begin{tikzcd}
	\lat A  \arrow[r, hookleftarrow] & \mcm A \arrow[r] & 
	\smcm A \arrow[r, "\sim"] 
	& \dsing(A)\,.
\end{tikzcd}
\]
This method might be useful for other Gorenstein algebras with a manageable, for instance, representation-tame, category of lattices.

\subsection*{The Auslander-Reiten theory of lattices}

Let $\lat A$ denote the full subcategory of finitely generated  $A$-modules which are free, or, equivalently, \emph{maximal Cohen--Macaulay over $R$} in the sense of \ref{chunk:mcmR}.
In different terms, these are precisely the \emph{lattices over the $R$-algebra $A$}. We will use the latter terminology in order to avoid confusion with the later notion of \emph{maximal Cohen--Macaulay $A$-modules}.

\begin{chunk}\label{ch:R-duality}
	According to \cite[Proposition I.7.3]{Auslander:1978a}, there is an exact duality
	\begin{align*}
		(-)^* \colonequals  \Hom_R(-,R) \colon  {\lat {\op A}} \longiso {\lat A}\,.
	\end{align*}
	Moreover, there is an isomorphism of functors 
	\begin{align}
		\label{eq:Nak-lat}
		\Hom_A(-,A)^* \cong
		\omega_{A/R} \otimes_A -\colon 
		{\proj (A)} \longiso {\inj {\lat A}} 
	\end{align}
	which yield an equivalence
	between the category of finitely generated projective 
	$A$-modules and the full subcategory of injective objects in the exact category of $A$-lattices.
\end{chunk}

Lemma~\ref{le:global-gentle}~\eqref{le:gg3} states that
$A$ is a finite free $R$-algebra such that the $Q(R)$-algebra $A \otimes_R Q(R)$ is semisimple.
This has the following consequences.
\begin{lemma} \label{lemma:semisimple} In the setup above, the following statements hold.
	\begin{enumerate}
		\item \label{eq:submodules} An $A$-module $L$
		is an $A$-lattice if and only if there is an $A$-linear embedding of $L$ into a finitely generated projective $A$-module.
		\item \label{eq:Ext-finite}
		For any finitely generated $A$-modules $M$ and $N$
		and any integer $i > 0$ 
		the $R$-module $\Ext^i_A(M,N)$ is a torsion module.
	\end{enumerate}
\end{lemma}
The first statement is known by \cite[Exercise~23.1]{Curtis--Reiner:1990},
the second follows from the proof of \cite[Proposition~25.5]{Curtis--Reiner:1990}. 
We give the arguments for the sake of completeness.
\begin{proof}
	In the following, $Q(R)$ is abbreviated to $Q$.
	\begin{enumerate}
		\item 
		Assume that $L$ is an $A$-lattice. Since $A \otimes_R Q$ is self-injective, $L \otimes_R Q$ embeds into $(A \otimes_R Q)^n$ for certain $n \geq 0$.
		Let $\iota$ denote the $A$-linear composition
		$$
		\begin{tikzcd}[ampersand replacement=\&]
			L \ar[hookrightarrow]{r}{\eta_L} \& L \otimes_R Q \ar[hookrightarrow]{r} \& (A \otimes_R Q)^n \longiso A^n \otimes_R Q\,.\end{tikzcd}
		$$
		Let $x_1,x_2 \ldots x_m$ be generators of $L$.
		For any $1 \leq i \leq m$ we may write $\iota(x_i)$
		as $\sum_{j=1}^n \sum_{k=1}^{\ell_j}
		a_{ijk} \otimes_R {\frac{r_{ijk}}{s_{ijk}}}$.
		Set $s$ to be the product of all denominators $s_{ijk}$.
		Since $\iota(Ls) \subseteq \Im \eta_{A^n}$, there is an $A$-linear embedding of $L \cong Ls$ into $A^n$.
		
		Vice versa, the category $\lat A$ contains $\proj(A)$ and is closed under submodules, because $A$ is finite and projective over the Dedekind domain $R$.
		\item In the notations above, 
		semisimplicity of $A \otimes_R Q$ implies that
		\[ 
		\Ext^i_A(M,N) \otimes_R Q \cong \Ext^i_{A \otimes_R Q}(M \otimes_R Q,N \otimes_R Q) = 0\,. 
		\qedhere
		\]
	\end{enumerate}
\end{proof}
More importantly, the category ${\lat A}$ admits almost-split sequences by results of Auslander~\cite{Auslander:1978a}. 
The next statement allows to compute these using the dualising bimodule $\omega_{A/R}$ of the algebra $A$.
\begin{proposition}\label{proposition:AR-translate}
	Let $L$ be a non-projective $A$-lattice. 
	Choose a projective cover $\pi\colon P_1 \twoheadrightarrow L$ and an $A$-linear monomorphism 
	$\iota\colon L \hookrightarrow P_0$ into a finitely generated projective $A$-module.
	Then there is 
	an isomorphism
	$
	\tau(L) \cong \Ker (\omega_{A/R} \otimes_A \partial_1)
	$
	of $A$-lattices, where
	$\partial_1$ denotes the composition $\iota  \pi \colon P_1 \lra P_0$.
\end{proposition}
\begin{proof}
	Applying $(-,A) \colonequals \Hom_A(-,A)$ to a minimal projective presentation of $L$
	\begin{align*}
		\begin{tikzcd}[ampersand replacement=\&]
			P_2 \ar{r}{\partial_2} \& P_1 \ar{r}{\pi} \ar{r}\& L \ar{r} \& 0
		\end{tikzcd}
	\end{align*}
	yields an exact sequence of $\op A$-modules
	\[
	\begin{tikzcd}[ampersand replacement=\&]
		0 \ar{r} \& (L,A) \ar{r}{(\pi,A)} \& 
		(P_1,A) \ar{r}{(\partial_2,A)} \&  (P_2,A)
		\ar{r} \& {\Tr {L}} \ar{r} \& 0
	\end{tikzcd}
	\]
	with  ${\Tr {L}} \colonequals 
	\Coker{(\partial_2,A)}$
	and $
	\syz{(\Tr {L})} \cong \Coker{(\pi,A)}$.
	According to \cite[Proposition~I.8.8]{Auslander:1978a} the Auslander-Reiten translation of $L$ is given by
	\[
	{ \tau(L) \colonequals (\syz(
		{\Tr {L}}))^* = \Ker ((\pi,A)^*)\,, \qquad\text{where $(-)^* \colonequals  \Hom_R(-,R)$}\,.}
	\]
	Since $\syz{(\Tr {L})}$ embeds into an $\op A$-lattice,
	it is an $\op A$-lattice as well, and thus 
	$(\pi,A)^*$ is surjective. 
	Lemma~\ref{lemma:semisimple}
	ensures the existence of an embedding $\iota\colon L \hookrightarrow P_0$ 
	and implies that
	$\Coker(\iota,A)=\Ext^1_A(\Coker \iota,A)$ is torsion over $R$. Therefore $(\Coker(\iota,A))^*$ is zero, and thus $(\iota,A)^*$ is injective.
	
	It follows that there is a commutative diagram with exact rows
	\begin{align*}
		\begin{tikzcd}[ampersand replacement=\&, column sep=1cm]
			0 \ar{r} \& \tau(L) \ar[dashed]{d}{\phi}[swap,anchor=south, rotate=90, inner sep=.5mm]{\sim} \ar{r} \& (P_1,A)^* \ar{r}{(\pi,A)^*} \ar{d}{\eta_{P_1}}[swap,anchor=south, rotate=90, inner sep=.5mm]{\sim} \& (L,A)^* \ar{r} \ar[hookrightarrow]{d}{\eta_{P_{0}} \cdot (\iota,A)^*} \& 0 \\
			0 \ar{r} \& \Ker (\omega_{A/R} \otimes_A \partial_1) \ar{r} \&
			\omega_{A/R} \otimes_A P_1 \ar{r}{\omega_{A/R} \otimes_A \partial_1} \& \omega_{A/R} \otimes_A P_{0}  \& 
		\end{tikzcd}
	\end{align*}
	where $\eta_{P_1}$ and $\eta_{P_{0}}$ are specialisations of the isomorphism
	of functors~\eqref{eq:Nak-lat}.
	The Snake Lemma
	implies bijectivity of the induced $A$-linear map $\phi$.
\end{proof}

\subsection*{From vertices and arrows to lattices}
While the statements of the previous subsection extend to a more general context, the next results concern the specific combinatorics of lattices over the gentle Gorenstein $R$-algebra $A$.

By~\eqref{eq:Nak-lat}, each vertex $j \in Q_0$ gives rise to an indecomposable projective $A$-module
$P_j =  A e_j$ as well as an indecomposable $A$-lattice
\[
I_j \colonequals  
\omega_{A/R} \otimes_A P_j \cong 
\begin{cases}
	\rad {P_j}& \text{ if $j$ is a transition vertex,}\\
	\quad P_j & \text{ if $j$ is $2$-regular,}
\end{cases}
\]
where the last 
isomorphism is a special case of Corollary~\ref{corollary:Nakayama-functor}.

The arrows of $Q$ give rise to lattices with modest properties as well.
According to previous notation, for an $A$-lattice $L$ we denote by 
$\underline{\End}_{A}(L)$ the quotient of $A$-linear endomorphisms of $L$ modulo the ideal of endomorphisms which factor through a projective $A$-module.
\begin{lemma} \label{lemma:arrow-ideals}
	For any arrow $\alpha \in Q_1$ the following statements hold.
	\begin{enumerate}
		\item \label{le:ai1}
		The left ideal $L_\alpha \colonequals  A \alpha$ is an indecomposable $A$-lattice.
		\item \label{le:ai2} For any arrow $\beta \in Q_1$ any non-isomorphism $f\colon L_\beta \to L_{\alpha} $ factors through the projective cover $\pi\colon P_{t(\alpha)} \to L_{\alpha}$, $p \mapsto p \alpha$.
		\item \label{le:ai3} It holds that $\End_{A}(L_{\alpha}) \cong R$ and $\underline{\End}_{A}(L_{\alpha}) \cong k$.
		\item \label{le:ai4} The lattice $L_{\alpha}$ is projective if and only if $t(\alpha)$ is a transition vertex. 
	\end{enumerate}
\end{lemma}
\begin{proof} The claims are consequences of the following observations.
	
	\eqref{le:ai1} 
	The category $\lat A$ is closed under submodules.
	
	\eqref{le:ai2} The map $f$ is given by the right multiplication with $f(\beta) \in e_{t(\beta)} A \alpha$. 
	
	\eqref{le:ai3}
	In case $\alpha = \beta$, it holds that $f(\alpha) \in k\llbracket c_{\alpha} \rrbracket$.
	
	\eqref{le:ai4}
	The epimorphism $\pi$ splits if and only if $\Ker \pi = 0$.
\end{proof}
Next, we compute the Auslander-Reiten translation $\tau$ on all non-projective arrow ideals.
\begin{lemma}\label{lemma:AR-sequences}
	Let $\alpha \in Q_1$ such that $t(\alpha)$ is not a transition vertex.
	Let $\varrho(\alpha)$ denote the unique arrow in $(Q,I)$ 
	with $s(\varrho(\alpha)) = t(\alpha)$ and $\varrho(\alpha) \alpha \in I$.
	Then the sequence
	\begin{align}
		\label{eq:syzygy-sequence}
		\begin{tikzcd}[ampersand replacement=\&] 0 \arrow[r] \& L_{\varrho(\alpha)} \arrow[r] \& P_{t(\alpha)} \arrow[r, "\pi"] \& L_{\alpha} \arrow[r] \& 0 \end{tikzcd}
	\end{align}
	is an almost-split sequence in the category of $A$-lattices.
	So $\tau(L_{\alpha}) \cong L_{\varrho(\alpha)} \cong \Omega(L_\alpha)$.
\end{lemma}
\begin{proof}
	Since $t(\alpha)$ is $2$-regular, the arrow $\beta \colonequals \varrho(\alpha)$ is well-defined
	and  $L_{\alpha}$ is a non-projective $A$-lattice by Lemma~\ref{lemma:arrow-ideals}. 
	The relations in $(Q,I)$
	yield that $\Ker \pi = L_{\beta}$. 
	Since the composition $\partial_1$ of the projective cover $\pi\colon P_{t(\alpha)} \twoheadrightarrow L_{\alpha}$ 
	with the embedding $L_{\alpha} \hookrightarrow P_{s(\alpha)}$ is given by the right multiplication
	with $\alpha$,
	Proposition~\ref{proposition:AR-translate}
	and Corollary~\ref{corollary:Nakayama-functor}
	imply that
	\begin{align*}
		\begin{tikzcd}[ampersand replacement=\&]
			\tau(L_{\alpha}) \cong 
			\Ker (\omega_{A/R}\otimes_A \partial_1 ) \cong \Ker(
			P_{t(\alpha)} \ar{r}{\cdot {(\pm \alpha)}}
			\& I_{s(\alpha)})
			\cong L_{\beta} \cong \Omega(L_{\alpha})\,.
		\end{tikzcd}
	\end{align*}
	The Auslander-Reiten formula yields that $\Ext^1_A(L_{\alpha},\tau(L_{\alpha} )) \cong \underline{\End}_A(L_{\alpha})^{\vee} \cong k$.
	This implies that any non-split exact sequence 
	$0 \to \tau(L_{\alpha}) \to  E \to L_{\alpha} \to 0$ of $A$-lattices, and thus
	\eqref{eq:syzygy-sequence} are
	almost-split. 
\end{proof}

The last lemma suggests to compute the projective resolutions of arrow ideals.

\begin{chunk}\label{chunk:projective-resolutions}
	Any arrow $\alpha\in Q_1$ is the beginning of a unique path $d_\alpha \colonequals \alpha_n \ldots \alpha_2 \alpha_1$, introduced in \ref{ch:differential},
	which determines the
	projective resolution 
	of $L_{\alpha}$ 
	as follows.
	\begin{enumerate}
		\item If $d_{\alpha}$ is a differential cycle, 
		the projective resolution of $L_{\alpha}$ is periodic:
		\[
		\begin{tikzcd}
			\ldots
			P_{t(\alpha_1)} \ar{r}{\cdot \alpha_1} &	
			P_{t(\alpha_n)} \ar{r}{\cdot \alpha_n} & 
			\quad \ldots
			P_{t(\alpha_{3})} \ar{r}{\cdot \alpha_{3}} & 
			P_{t(\alpha_2)} \ar{r}{\cdot \alpha_2} & P_{t(\alpha_1)} \end{tikzcd}
		\]
		\item Otherwise, $d_{\alpha}$ is a differential walk and the projective resolution of $L_{\alpha}$ is finite:
		\[
		\begin{tikzcd}
			\quad \  \ldots
			0 \ar{r} &	P_{t(\alpha_n)} \ar{r}{\cdot \alpha_n} &
			\quad \ldots P_{t(\alpha_{3})} \ar{r}{\cdot \alpha_3} &  P_{t(\alpha_2)} \ar{r}{\cdot \alpha_2} & P_{t(\alpha_1)} \end{tikzcd}
		\]
	\end{enumerate}
\end{chunk}

Adding the syzygies to projective resolutions of arrow ideals and certain radical embeddings yields the complete picture of the category $\lat A$.
In particular, any indecomposable $A$-lattice is projective or an arrow ideal.
\begin{proposition}\label{proposition:AR-quiver}
	The Auslander-Reiten quiver of $A$-lattices is given by the syzygy sequences of the form \eqref{eq:syzygy-sequence}
	together with the inclusion ${\rad {P_j}} \hookrightarrow P_j$ for each transition vertex $j$ of $Q$.
\end{proposition}
\begin{proof}
	For simplicity of the presentation, assume that the quiver $Q$ is connected.
	Each differential cycle in $(Q,I)$ induces a periodic $\tau$-orbit in the Auslander-Reiten quiver $\Gamma_{\lat A}$.
	Each differential walk between transition vertices 
	gives rise to a full finite $\tau$-orbit.
	Completing these orbits by radical embeddings of indecomposable projectives and projective covers yields a full connected component $\mathcal{C}$ of the quiver $\Gamma_{\lat A}$.
	A result of Wiedemann \cite{Wiedemann:1981}
	implies that $\mathcal{C} = \Gamma_{\lat A}$.
\end{proof}
\begin{example}
	The Auslander-Reiten quiver of $A$-lattices in Example~\ref{example:prototype} has three finite and two periodic $\tau$-orbits, indicated by the dashed arrows:
	
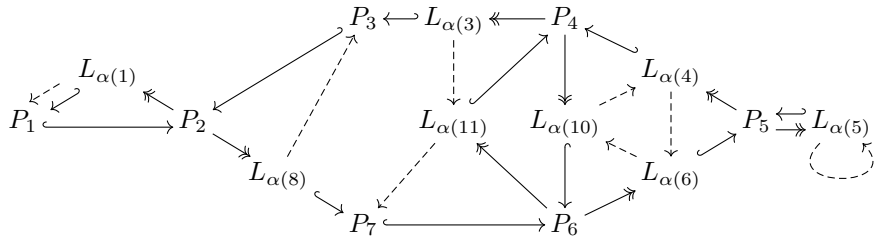
\[
	{\begin{tikzcd}[row sep=0.25cm, column sep=0.47cm, cells={outer sep=1pt, inner sep=1pt}]
			& & & 
			& P_3 \ar[hookleftarrow]{r} & L_{\alpha(3)} \ar[dashed]{dd} & P_4 \ar[twoheadrightarrow]{l} \ar[hookleftarrow]{rd} \ar[twoheadrightarrow]{dd}
			\\
			& L_{\alpha(1)} \ar[yshift=3pt, xshift=-5pt, dashed]{ld} & &
			& & & 
			& L_{\alpha(4)} \ar[dashed]{dd}
			\\
			P_1  \ar[yshift=-2pt, hookrightarrow]{rr} \ar[xshift=3pt, hookleftarrow]{ru} & &  P_2 \ar[hookleftarrow]{rruu} \ar[twoheadrightarrow]{rd} \ar[twoheadrightarrow]{lu} &
			& & L_{\alpha(11)} \ar[dashed]{ldd} \ar[hookrightarrow]{ruu} &
			L_{\alpha(10)} \ar[hookrightarrow]{dd} \ar[dashed]{ru} & & P_5 \ar[yshift=-3pt, twoheadrightarrow]{r} \ar[twoheadrightarrow]{lu} \ar[yshift=3pt, hookleftarrow]{r} & 
			L_{\alpha(5)} \ar[looseness=4, out=-135, in=-45,dashed] 
			\\
			&&& L_{\alpha(8)} \ar[dashed]{ruuu} \ar[hookrightarrow]{rd} 
			&&&
			& L_{\alpha(6)} \ar[hookrightarrow]{ru} \ar[dashed]{lu}
			\\
			&&&
			& P_7   \ar[hookrightarrow]{rr} && P_6 \ar[twoheadrightarrow]{luu} \ar[twoheadrightarrow]{ru}
			&& \end{tikzcd}
	}
	\]
\end{example}
\begin{remark}
	It is also possible to apply the Drozd--Kiri\v{c}enko Rejection Lemma \cite{Drozd--Kirichenko:1972} (see also \cite{Drozd:2020}) to derive Proposition~\ref{proposition:AR-quiver}.
	Alternatively, one may determine the indecomposable $A$-lattices using the works \cite{Green--Reiner:1978} or \cite{Ringel--Roggenkamp:1979} on \emph{B\"ackstr\"om orders}, and then compute their Auslander-Reiten quiver via a method by Roggenkamp \cite{Roggenkamp:1983}. 
\end{remark}

\subsection*{Maximal Cohen--Macaulay modules and the singularity category}

According to Lemma~\ref{le:mcm-base-independence}, we may define \emph{maximal Cohen--Macaulay $A$-modules} as $A$-lattices $M$ satisfying 
$\Ext^i_A(M,A) = 0$ for $i \geq 1$.
As in Section~\ref{se:maximal-cohen-macaulay-modules}, we denote by $\smcm A$ the stable category of such modules.

The projective resolutions computed in  \ref{chunk:projective-resolutions}
can be used to show the following.
\begin{lemma}\label{lemma:smcm-arrows}
	An arrow $\alpha \in Q_1$ appears in differential cycle if and only if its left ideal $L_{\alpha}$ is a non-projective maximal Cohen--Macaulay $A$-module. \qed
\end{lemma}

According to Propositions~\ref{pr:fp}~and~\ref{pr:gentle-gorenstein}, $A$ is a Gorenstein $R$-algebra. 
Moreover, Lemma~\ref{le:global-gentle}~\eqref{le:gg3} implies that $A$ has an isolated singularity.
In particular, the dualising bimodule of $A$ gives rise to a Serre functor on $\dsing(A)$ by Theorem~\ref{theorem:serre-duality}.
\begin{theorem}\label{th:dsing-gentle}
	The singularity category ${\dsing(A)}$ of the gentle Gorenstein $R$-algebra $A$ has the following description.
	\begin{enumerate}
		\item \label{th:sd1} There is a bijection between the set $Q_1^{dc}$ of all arrows in $Q$ which appear in differential cycles and the set of isomorphism classes of indecomposable objects in the category ${\dsing(A)}$ given by
		\[
		Q_1^{dc} \longiso {\ind {{\dsing(A)}}}\,, \qquad \alpha \longmapsto L_\alpha\,,
		\]
		where each arrow ideal $L_{\alpha}$ is viewed as an object in $\dsing(A)$.
		\item \label{th:sd2} For any $\alpha, \beta \in Q_1^{dc}$ it holds that $\dim_k \Hom_{\dsing(A)}(L_{\alpha}, L_\beta) = \delta_{\alpha\beta}$.
		\item \label{th:sd3} For any arrow $\alpha \in Q_1^{dc}$ there are isomorphisms
		\[
		\omega_{A/R} \lotimes_A L_\alpha\cong L_{\alpha}\quad\text{and} \quad \Sigma^{-1}(L_\alpha) \cong L_{\varrho(\alpha)} \text{ in }\dsing(A)\,,
		\]
		where $\varrho(\alpha)$ denotes the unique arrow with 
		$s(\varrho(\alpha)) = t(\alpha)$
		and
		$\varrho(\alpha) \alpha \in I$.
	\end{enumerate}
\end{theorem}

\begin{proof}
	\eqref{th:sd1}
	Proposition~\ref{proposition:AR-quiver}	and Lemma~\ref{lemma:smcm-arrows} describe all indecomposable objects in the first three categories up to isomorphism:
	\[
	\begin{tikzcd}[column sep=1cm]
		{\lat{A}} \ar[hookleftarrow]{r} & \mcm{A}  \ar{r} & {\smcm{A}}  \ar{r}{\sim}[swap,yshift=-4pt]{\text{Thm.}\ref{th:buchweitz}} & {\dsing(A)}\,.
	\end{tikzcd}
	\]
	
	\eqref{th:sd2} follows using the diagram above together with Lemma~\ref{lemma:arrow-ideals}~\eqref{le:ai2} and \eqref{le:ai3}.
	
	\eqref{th:sd3}
	Since $\omega_{A/R} \lotimes_R -$ induces the Serre functor on $\dsing(A)$,
	the previous statement \eqref{th:sd2} yields 
	the first isomorphism.
	The second isomorphism follows 
	from Lemma~\ref{lemma:AR-sequences}
	and the fact 
	that the shift functor $\Sigma$ on ${\smcm A}$ is isomorphic to $\syz^{-1}$ (see, for example, \cite[Theorem 6.2.5]{Krause:2021}). 
\end{proof}

As observed by Kalck for finite dimensional gentle algebras \cite{Kalck:2015}, the semisimple category $\dsing(A)$ still contains some useful homological information: the number and the lengths of its shift orbits.
In particular, any two derived equivalent gentle Gorenstein algebras have the same number and lengths of differential cycles.
\begin{example}
	The singularity category of the quiver from Example~\ref{example:prototype} is given by the shift orbit of $L_{\alpha(4)}$ which has length three 
	and the shift-invariant object $L_{\alpha(5)}$.
\end{example}

\begin{remark}
	Any gentle quiver gives rise to a graph which is embedded into an oriented surface with certain additional data.
	In this context, the differential cycles of the gentle quiver may be viewed as certain faces of its embedded graph. 
	
	In \cite{Palu--Pilaud--Plamondon:2019} Palu, Pilaud and Plamondon provide an explicit bijection between the more general class of locally gentle quivers and embedded graphs of certain oriented surfaces. 
	Koszul duality of two gentle quivers translates into a natural duality of their associated graphs, see \cite{Opper--Plamondon--Schroll:2018} and \cite{Palu--Pilaud--Plamondon:2019}.
	
	Finite dimensional gentle algebras have been classified up to derived equivalence by Amiot, Plamondon and Schroll \cite{Amiot--Plamondon--Schroll:2019} and Opper \cite{Opper:2019} using a geometric model of the corresponding surface elaborated by Opper, Plamondon and Schroll \cite{Opper--Plamondon--Schroll:2018}.
	We refer to the introductions therein for a broader overview on gentle algebras.
\end{remark}

\begin{ack}
We would  like to acknowledge our intellectual debt to Ragnar-Olaf Buchweitz, whose pioneering monograph \cite{Buchweitz:1986} served as a major inspiration for this project.  The motto of that work: \emph{Guided, but not led, by commutative algebra}, may well be a subtitle to \cite{Iyengar--Krause:2021} and to this manuscript.    We  thank our collaborators Luchezar Avramov, Dave Benson,  John Greenlees,  and Julia Pevtsova, for freely sharing their insights on Gorenstein phenomena and duality.  Our thanks  also to Janina Letz for comments and corrections on an earlier version of this manuscript.
The authors are very grateful to an anonymous referee for important corrections and useful suggestions on the exposition of this paper.
The first author would like to thank the Algebra group of the Ruhr University Bochum for a stimulating research environment.
The second author  thanks the organisers of the ICRA workshop for giving an opportunity to presenting this work in that forum, and the editors of this volume for their patience.
\end{ack}

\begin{funding}
WG was fully supported by the Ruhr University Bochum during his work on the present article.
SBI was partly supported by NSF grant DMS-2001368.
\end{funding}


\begin{thebibliography}{10}
	
	\bibitem{Abe--Hoshino:2007}
	H.~Abe\ and\ M.~Hoshino, Derived equivalences and Gorenstein algebras, J. Pure Appl. Algebra {\bf 211} (2007), no.~1, 55--69. 
	
	\bibitem{Amiot--Plamondon--Schroll:2019}
	C.~Amiot, P.-G.~Plamondon,\ and\ S.~Schroll, A complete derived invariant for gentle algebras via winding numbers and Arf invariants,
	preprint 2019, \href{https://arxiv.org/abs/1904.02555}{\texttt{arXiv:1904.02555}}.
	
	\bibitem{Assem--Skowronski:1987} I.~Assem\ and\ A.~Skowro\'{n}ski, Iterated tilted
	algebras of type $\tilde{\bf A}_n$, Math. Z. {\bf 195} (1987),
	no.~2, 269--290.
	
	\bibitem{Auslander:1978a} M.~Auslander, Functors and morphisms determined by
	objects, in {\it Representation theory of algebras (Proc. Conf.,
		Temple Univ., Philadelphia, Pa., 1976)}, 1--244. Lecture Notes in
	Pure Appl. Math., 37, Dekker, New York, 1978.
	
	\bibitem{Auslander--Reiten:1991} M.~Auslander\ and\ I.~Reiten,
	Cohen--Macaulay and Gorenstein Artin algebras, in {\it Representation
		theory of finite groups and finite-dimensional algebras
		(Bielefeld, 1991)}, 221--245, Progr. Math., 95, Birkh\"{a}user,
	Basel, 1991.
	
	\bibitem{Avramov--Iyengar:2021} L.~L.~Avramov\ and\ S.~B.~Iyengar,
	Homological dimensions over Noether algebras, appendix to
	{\it Support varieties and modules of finite projective dimension
		for modular Lie superalgebras}, Algebra Number Theory {\bf 15}  (2021), no.~5.
	
	\bibitem{Avramov--Iyengar--Lipman:2010}
	L.~L.~Avramov, S.~B.~Iyengar,\ and\ J.~Lipman,  Reflexivity and rigidity for complexes. I. Commutative rings. Algebra Number Theory {\bf 4} (2010), no.~1, 47--86.
	
	\bibitem{Bass:1963} H.~Bass, On the ubiquity of Gorenstein rings,
	Math. Z. {\bf 82} (1963), 8--28.
	
	\bibitem{Bass:1968} H.~Bass, {\it Algebraic $K$-theory},
	W. A. Benjamin, Inc., New York-Amsterdam, 1968.
	
	\bibitem{Bennett-Tennenhaus:2017}
	R.~Bennett-Tennenhaus, Functorial filtrations for semiperfect generalisations of gentle algebras, Ph.D. thesis, University of Leeds 2017,
	\url{http://etheses.whiterose.ac.uk/19607/}
	
	\bibitem{Benson--Iyengar--Krause:2008a} D.~J.~Benson, S.~B.~Iyengar,\ and\
	H.~Krause, Local cohomology and support for triangulated categories,
	Ann. Sci. \'{E}c. Norm. Sup\'{e}r. (4) {\bf 41} (2008), no.~4,
	573--619.
	
	\bibitem{Benson--Iyengar--Krause:2011a} D.~J.~Benson, S.~B.~Iyengar,\ and\
	H.~Krause, Stratifying triangulated categories, J. Topol. {\bf 4}
	(2011), no.~3, 641--666.
	
	\bibitem{Benson--Iyengar--Krause--Pevtsova:2019a} D.~J.~Benson,
	S.~B.~Iyengar, H.~Krause,\ and\ J.~Pevtsova,
	Local duality for representations of finite group schemes,
	Compos. Math. {\bf 155} (2019), no.~2, 424--453.
	
	\bibitem{Benson--Iyengar--Krause--Pevtsova:2020a} D.~J.~Benson,
	S.~B.~Iyengar, H.~Krause,\ and\ J.~Pevtsova, Local duality for
	the singularity category of a finite dimensional Gorenstein
	algebra, Nagoya Math. J. {\bf 244}, (2021), 1--24.
	
	\bibitem{Benson--Iyengar--Krause--Pevtsova:2022} 
	D.~J.~Benson, S.~B.~Iyengar, H.~Krause,\ and\ J.~Pevtsova, Fibrewise stratification of group representations, preprint 2022, 
	\href{https://arxiv.org/abs/2204.10431v3}{\texttt{arXiv:2204.10431v3}.}
	
	\bibitem{Bruns--Herzog:1998a} W.~Bruns\ and\ J.~Herzog, {\it
		Cohen--Macaulay rings}, Cambridge Studies in Advanced
	Mathematics, 39, (second ed.) Cambridge University Press, Cambridge, 1998.
	
	\bibitem{Buchweitz:1986} 
	R.-O.~Buchweitz, 
	Maximal Cohen--Macaulay modules and Tate-cohomology,  Mathematical Surveys and Monographs, 262, Amer. Math. Soc. 2021; published version of  \url{https://tspace.library.utoronto.ca/handle/1807/16682} 
	with appendices by L.~L.~Avramov, B.~Briggs, S.~B.~Iyengar, and J.~C.~Letz. 
	
	\bibitem{Burban--Drozd:2004} I.~Burban\ and\ Yu.~A.~Drozd, Derived categories of nodal algebras, J. Algebra {\bf 272} (2004), no.~1, 46--94. 
	
	\bibitem{Burban--Drozd:2017} I.~Burban\ and\ Yu.~A.~Drozd, On the derived categories of gentle and skew-gentle algebras: homological algebra and matrix problems, preprint 2017, \href{https://arxiv.org/abs/1706.08358}{\texttt{arXiv:1706.08358}}
	
	\bibitem{Burban--Drozd:2019} I.~Burban\ and\ Yu.~A.~Drozd,
	Non-commutative nodal curves and derived tame algebras, 
	preprint 2019, \href{https://arxiv.org/abs/1805.05174}{\texttt{arXiv:1805.05174}}
	
	\bibitem{Curtis--Reiner:1990}
	C.~W.~Curtis\ and\ I.~Reiner, {\it Methods of representation theory. Vol. I. With applications to finite groups and orders.} Reprint of the 1981 original. Wiley Classics Library. A Wiley-Interscience Publication. John Wiley \& Sons, Inc., New York, 1990. 
	
	\bibitem{Drozd:2020} Yu.~A.~Drozd, Rejection lemma and almost split sequences, 
	Ukr. Math. J. {\bf 73:6} (2021), 780--798.
	
	\bibitem{Drozd--Kirichenko:1972} Yu.~A.~Drozd\ and\ V.~V.~Kiri\v{c}enko, On quasi-Bass orders. Izv. Akad. Nauk SSSR. Ser. Mat., {\bf 36} (1972), 328--370. 
	
	\bibitem{Dwyer--Greenlees--Iyengar:2005} W.~G.~Dwyer,
	J.~P.~C.~Greenlees,\ and\ S.~Iyengar, Duality in algebra and
	topology, Adv. Math. {\bf 200} (2006), no.~2, 357--402.
	
	\bibitem{Eisele:2012}
	F.~Eisele, $p$-adic lifting problems and derived equivalences. J. Algebra {\bf 356} (2012), 90--114. 
	
	\bibitem{Eisele:2021} F.~Eisele, Bijections of silting complexes and derived Picard groups, 
	J. Lond. Math. Soc. (2) {\bf 106} (2022), 1008--1060.
		
	\bibitem{Erdmann--Holloway--Snashall--Solberg--Taillefer:2004} K.~Erdmann, M.~Holloway, N.~Snashall, \O.~Solberg,\ and\ R.~Taillefer, Support varieties for
	selfinjective algebras, $K$-Theory {\bf 33} (2004), no.~1, 67--87.
	
	\bibitem{Formanek:1973}
	E.~Formanek, Faithful Noetherian modules. Proc. Amer. Math. Soc. {\bf 41} (1973), 381--383.
	
	\bibitem{Fossum--Foxby--Griffith--Reiten:1975}
	R.~Fossum,  H.-B.~Foxby,  P.~Griffith,\ and\ I.~Reiten,  Minimal injective resolutions with applications to dualizing modules and Gorenstein modules.
	Inst. Hautes \'Etudes Sci. Publ. Math.  \textbf{45}  (1975), 193--215.
	
	\bibitem{Foxby:1979}
	H.-B.~Foxby,  Bounded complexes of flat modules. J. Pure Appl. Algebra \textbf{15} (1979), 149--172.  
	
	\bibitem{Frankild--Jorgensen:2003}
	A.~Frankild\ and\ P.~J{\o}rgensen,  Gorenstein differential graded algebras. Israel J. Math. {\bf 135} (2003), 327--353.
	
	\bibitem{Gabriel:1962} P.~Gabriel, Des cat\'{e}gories ab\'{e}liennes,
	Bull. Soc. Math. France {\bf 90} (1962), 323--448.
	
	\bibitem{Geiss--Reiten:2005} Ch.~Gei{\ss}\ and\ I.~Reiten, Gentle algebras are
	Gorenstein, in {\it Representations of algebras and related topics},
	129--133, Fields Inst. Commun., 45, Amer. Math.
	Soc., Providence, RI, 2005.
	
	\bibitem{Ginzburg:2006}
	V.~Ginzburg, Calabi-Yau algebras, preprint 2006, \href{https://arxiv.org/abs/math/0612139}{\texttt{arXiv:math/0612139}}
	
	\bibitem{Gnedin:2019} 
	W.~Gnedin, Calabi-Yau properties of ribbon graph orders, preprint 2019,
	\href{https://arxiv.org/abs/1908.08895}{\texttt{arXiv:1908.08895}}
	
	\bibitem{Gnedin:2022} W.~Gnedin, Silting theory under change of rings, preprint 2022, \href{https://arxiv.org/abs/2204.00608}{\texttt{arXiv:2204.00608}}
	
	\bibitem{Goto:1982} S.~Goto, Vanishing of
	${\rm Ext}\sp{i}\sb{A}(M,\,A)$, J. Math. Kyoto Univ. {\bf 22}
	(1982/83), no.~3, 481--484.
	
	\bibitem{Green--Reiner:1978}
	E.~L.~Green\ and\ I.~Reiner, Integral representations and diagrams. Michigan Math. J. {\bf 25} (1978), no.~1, 53--84. 
	
	\bibitem{Green--Zacharia:1994} E.~L.~Green\ and\ D.~Zacharia, The cohomology ring of
	a monomial algebra, Manuscripta Math. {\bf 85} (1994), no.~1,
	11--23.
	
	\bibitem{Higman:1959} D.~G.~Higman,  On isomorphisms of orders, Michigan Math. J. {\bf 6} (1959), 255--257. 
	
	\bibitem{Hopkins:1987}
	M.~J.~Hopkins, {Global methods in homotopy theory}, in \emph{Homotopy Theory, (Durham 1985)}, 73--96, 
	London Math. Soc. Lecture Note Ser., 117, Cambridge Univ. Press, Cambridge, 1987.
	
	\bibitem{Hubery--Krause:2016} A.~Hubery\ and\ H.~Krause, A
	categorification of non-crossing partitions,
	J. Eur. Math. Soc. (JEMS) {\bf 18} (2016), no.~10, 2273--2313.
	
	\bibitem{Ingalls--Thomas:2009} C.~Ingalls\ and\ H.~Thomas, Noncrossing
	partitions and representations of quivers, Compos. Math. {\bf 145}
	(2009), no.~6, 1533--1562.
	
	\bibitem{Iyama--Kimura:2021}
	O.~Iyama\ and\ Y.~Kimura, Classifying subcategories of modules over Noetherian algebras, preprint 2021, \href{https://arxiv.org/abs/2106.00469}{\texttt{arXiv:2106.00469}}
	
	\bibitem{Iyama--Reiten:2008}
	O.~Iyama\ and\ I.~Reiten, Fomin-Zelevinsky mutation and tilting modules over Calabi-Yau algebras. Amer. J. Math. {\bf 130} (2008), no.~4, 1087--1149.
		
	\bibitem{Iyengar--Krause:2021} S.~B.~Iyengar\ and\ H.~Krause, The
	Nakayama functor and its completion for Gorenstein algebras, Bull. Soc. Math. France {\bf 150} (2022), no.~2, 347--391.
	
	\bibitem{Kalck:2015} M.~Kalck, Singularity categories of gentle algebras, Bull. Lond. Math. Soc. {\bf 47} (2015), no.~1, 65--74. 
	
	\bibitem{Keller:1994}
	B.~Keller, Deriving DG categories, Ann. Sci. École Norm. Sup. (4) {\bf 27} (1994), no.~1, 63--102.
	
	\bibitem{Keller:2008}
	B.~Keller, Calabi-Yau triangulated categories, in {\it Trends in representation theory of algebras and related topics}, 467--489, EMS Ser. Congr. Rep., Eur. Math. Soc., Z\"urich, 2008. 
	
	\bibitem{Koenig--Zimmermann:1996}
	S.~K\"onig\ and\ A.~Zimmermann, Tilting hereditary orders, Comm. Algebra {\bf 24} (1996), no.~6, 1897--1913. 
	
	\bibitem{Krause:2021} H.~Krause, {\it Homological theory of representations}, 
	Cambridge Studies in Advanced Mathematics, Cambridge University Press,  195, Cambridge, 2022.
	
	\bibitem{Ladkani:2012} S.~Ladkani, On Jacobian algebras from closed surfaces, preprint 2012, \href{https://arxiv.org/abs/1207.3778}{\texttt{arXiv:1207.3778}}
	
	\bibitem{Lam:1991} T.~Y.~Lam, {\it A first course in noncommutative rings}, Graduate Texts in Mathematics, 131. Springer-Verlag, New York, 1991.
	
	\bibitem{Nagata:1975} M.~Nagata, {\it Local rings}, corrected reprint,
	Robert E. Krieger Publishing Co., Huntington, N.Y., 1975.
	
	\bibitem{Neeman:1992a} A.~Neeman,  The chromatic tower for $D(R)$, Topology {\bf 31}, no.~3,  (1992), 519--532.
	
	\bibitem{Opper:2019}
	S.~Opper, On auto-equivalences and complete derived invariants of gentle algebras,
	preprint 2019, \href{https://arxiv.org/abs/1904.04859}{\texttt{arXiv:1904.04859}}
	
	\bibitem{Opper--Plamondon--Schroll:2018}
	S.~Opper, P.-G.~Plamondon,\ and\ S.~Schroll, A geometric model for the derived category of gentle algebras, preprint 2018, \href{https://arxiv.org/abs/1801.09659}{\texttt{arXiv:1801.09659}}
	
	\bibitem{Orlov:2004a} D.~O.~Orlov, Triangulated categories of
	singularities and D-branes in Landau-Ginzburg models, Proc. Steklov
	Inst. Math. {\bf 2004}, no.~3(246), 227--248; translated from
	Tr. Mat. Inst. Steklova {\bf 246} (2004), Algebr. Geom. Metody,
	Svyazi i Prilozh., 240--262.
	
	\bibitem{Palu--Pilaud--Plamondon:2019}
	Y.~Palu, V.~Pilaud,\ and\ P.-G.~Plamondon, Non-kissing and non-crossing complexes for locally gentle algebras, J. Comb. Algebra {\bf 3} (2019), no.~4, 401--438.
	
	\bibitem{Reiten--Van-den-Bergh:2002}
	I.~Reiten\ and\ M.~Van den Bergh, Noetherian hereditary abelian categories satisfying Serre duality, J. Amer. Math. Soc. \textbf{15} (2002), 295--366.
	
	\bibitem{Rickard:1989}
	J.~Rickard,  Morita theory for derived categories. J. London Math. Soc. (2) {\bf 39} (1989), no.~3, 436--456.
	
	\bibitem{Rickard:1991b}
	J.~Rickard, Derived equivalences as derived functors. J. London Math. Soc. (2) {\bf 43} (1991), no.~1, 37--48.
	
	\bibitem{Rickard:1991}
	J.~Rickard,  Lifting theorems for tilting complexes. J. Algebra {\bf 142} (1991), no.~2, 383--393.
	
	\bibitem{Rickard:1998} J.~Rickard, Triangulated categories in the modular representation theory of finite groups, in {\it Derived equivalences for group rings}, 177--198, Lecture Notes in Math., 1685, Springer, Berlin, 1998. 
	
	\bibitem{Ringel--Roggenkamp:1979} C.~M.~Ringel\ and\ K.~W.~Roggenkamp, Diagrammatic methods in the representation theory of orders, J. Algebra \textbf{60} (1979), no.~1, 11--42. 
	
	\bibitem{Roberts:1980}
	P.~Roberts, \emph{Homological invariants of modules over commutative rings}, S\'eminaire de
	Math\'ematiques Sup\'erieures, \textbf{72}, Presses de l'Universit\'e de Montr\'eal, Montreal, Que., 1980. 
	
	\bibitem{Roberts:1987}
	P.~Roberts, Le th\'eor\'eme d'intersection, C. R. Acad. Sci. Paris S\'er. I Math. 304 (1987), no.~7, 177--180.
	
	\bibitem{Roggenkamp:1983} K.~W.~Roggenkamp, Auslander-Reiten species of B\"ackstr\"om orders, J. Algebra {\bf 85} (1983), no.~2, 449--476. 
	
	\bibitem{Skowronski--Yamagata:2011}
	A.~Skowro\'{n}ski and K.~Yamagata, {\it   Frobenius algebras. {I}. Basic representation theory.} EMS Textbooks in Mathematics, European Mathematical Society (EMS), Z\"{u}rich, 2011.
	
	\bibitem{Stevenson:2013} G.~Stevenson, Support theory via actions of
	tensor triangulated categories, J. Reine Angew. Math. {\bf 681}
	(2013), 219--254.
	
	\bibitem{Van-den-Bergh:2004}
	M.~Van den Bergh, Non-commutative crepant resolutions, in {\it The legacy of Niels Henrik Abel}, 749--770, Springer, Berlin, 2004. 
	
	\bibitem{Wiedemann:1981}  A.~Wiedemann, Brauer-Thrall ${\rm I}$ for orders and its application to orders with loops in their Auslander-Reiten graph, 
	in {\it Representations of algebras (Puebla, 1980)}, 350--357, Lecture Notes in Math., 903, Springer, Berlin-New York, 1981.   
	
	\bibitem{Zaks:1969} A.~Zaks, Injective dimension of semi-primary rings,
	J. Algebra {\bf 13} (1969), 73--86.
	
\end{thebibliography}
\end{document}